\documentclass[11pt,reqno]{amsart}
\usepackage{latexsym}
\usepackage{amssymb}
\usepackage{amsmath}
\usepackage{amscd}
\usepackage{amsthm}
\usepackage{graphicx}
\usepackage[hang]{caption}
\usepackage{indentfirst}
\usepackage{dsfont}
\usepackage{color}
\usepackage{bm}
\usepackage[section]{placeins}

\newcommand\dela[1]{}
\allowdisplaybreaks

\def\XXint#1#2#3{{\setbox0=\hbox{$#1{#2#3}{\int}$ }
\vcenter{\hbox{$#2#3$ }}\kern-.6\wd0}}

\newtheorem{theorem}{Theorem}
\numberwithin{equation}{section}
\numberwithin{theorem}{section}
\newtheorem{lemma}[theorem]{Lemma}
\newtheorem{cor}[theorem]{Corollary}

\newtheorem{remark}[theorem]{Remark}

\newtheorem{definition}[theorem]{Definition}

\numberwithin{equation}{section}

\newcommand{\mf}{\mathbf}

\newcommand{\wt}{\widetilde}
\newcommand{\e}{\varepsilon}
\def\R{\mathbb{ R}}
\def\Z{\mathbb{ Z}}
\def\E{\mathbb{E}}

\def\o{\overline}
\begin{document}

%\showkeys

\title{Homogenization of the evolution Stokes equation in a perforated domain with a stochastic Fourier boundary condition}

\author[H. Bessaih]{H. Bessaih}
\address{University of Wyoming, Department of Mathematics, Dept. 3036, 1000
East University Avenue, Laramie WY 82071, United States}
\email{ bessaih@uwyo.edu}

\author[Y. Efendiev]{ Y. Efendiev}
\address{Department of Mathematics \& ISC,
Texas A\&M University,
And,
Numerical Porous Media SRI Center, CEMSE Division,
King Abdullah University of Science and Technology, 
Thuwal 23955-6900,
Kingdom of Saudi Arabia}
\email{efendiev@math.tamu.edu}

\author[F. Maris]{ F. Maris}
\address{Numerical Porous Media SRI Center, CEMSE Division,
King Abdullah University of Science and Technology, 
Thuwal 23955-6900,
Kingdom of Saudi Arabia}
\email{florinmaris@gmail.com}

\maketitle

%\begin{center}
%{\large\bf Hom Stokes DBC\\
%\bigskip\bigskip
%.......................}
%\end{center}

{\footnotesize
\begin{center}

%Department...

\end{center}
}
\begin{abstract}
The evolution Stokes equation in a perforated domain subject to Fourier boundary condition on the boundaries of the holes is considered. We assume that the dynamic is driven by a stochastic perturbation on the interior of the domain and another stochastic perturbation on the boundaries of the holes.  The macroscopic (homogenized) equation is derived as another stochastic partial differential equation, defined in the whole non perforated domain. Here, the initial stochastic perturbation on the boundary becomes part of the homogenized equation as another stochastic force. We use the two-scale convergence method after extending the solution with 0 in the wholes to pass to the limit. By It\^o stochastic calculus, we get uniform estimates on the solution in appropriate spaces. In order to pass to the limit on the boundary integrals, we rewrite them in terms of  integrals  in the whole domain. In particular, for the stochastic integral on the boundary, we combine the previous idea of rewriting it on the whole domain with the assumption that the Brownian motion is of trace class. Due to the particular boundary condition dealt with, we get that the solution of the stochastic homogenized equation is not divergence free. However, it is coupled with the cell problem that has a divergence free solution. This paper represents an extension of the results from of Duan and Wang (Comm. Math. Phys. 275:1508--1527, 2007), where a reaction diffusion equation with a dynamical boundary condition with a noise source term on both the interior of the domain and on the boundary was studied, and through a tightness argument and a pointwise two scale convergence method the homogenized equation was derived.
\end{abstract}

\maketitle

\section{Introduction and formulation of the problem }
In this paper, we are interested in a fluid flow where the advective inertial forces are small compared with viscous forces. Starting from the evolution Stokes equations in a periodic porous medium, with a dynamical boundary condition driven by a noise source on the solid pores, the homogenized dynamic is rigorously recovered by the use of two-scale convergence method.

The homogenization of the Stokes problem in perforated domains goes back to S{\'a}nchez-Palencia in \cite{SP80} where an asymptotic expansion method was used. Rigorous proofs were given later by Tartar in \cite{T80} by the energy method and by Allaire in \cite{A92} where two scale convergence method was used. The two scale convergence was first introduced by Nguetseng in 1989 in \cite{N89} and later developed by Allaire in \cite{A-2s}. The idea of this method was to give a rigorous justification to the asymptotic expansion method. A more general setting has been defined by Nguetseng in \cite{N03}, \cite{N04} and later in \cite{NSW10}. The theory of the two scale convergence from the periodic to the stochastic setting has been extended by Bourgeat, A. Mikeli{\'c} and Wright in \cite{BMW94}, using techniques from ergodic theory. There is a vast literature for partial differential equations with random coefficients, where this method was used, however most of the tackled problems in this setting are not in perforated domains, (see \cite{BMW94}, \cite{BLL06} and the references therein). Much less was done for homogenization of stochastic partial differential equations, in particular in perforated domains. We mention the paper \cite{WD07} where a reaction diffusion equation with a dynamical boundary condition with a noise source term on both the interior of the domain and on the boundary was studied, and through a tightness argument and a pointwise two scale convergence method the homogenized equation was derived. A comprehensive theory for solving stochastic homogenization problems has been constructed recently in \cite{JL13} and \cite{RSW12}, where a $\Sigma$-convergence method adapted to stochastic processes was developed. An application of the method to the homogenization of a stochastic Navier-Stokes type equation with oscillating coefficients in a bounded domain (without holes) has been provided.

For the deterministic Stokes or Navier Stokes equations in perforated domains we refer to: \cite{SP80}, \cite{T80}, \cite {A92}, \cite{Mik91}, \cite{CDE96}, \cite{A91}, \cite{Mik95}. In \cite{CDE96} the Stokes problem in a perforated domain with a nonhomogeneous Fourier boundary condition on the boundaries of the holes was studied while in \cite{A91} the same problem was studied with a slip boundary condition. As far as we know, the stochastic Stokes or Navier Stokes equation in a perforated domain has not been studied. 

In this paper we consider a stochastic linear Navier Stokes equation in a periodically perforated domain with a noise source. On the boundaries of the holes, we consider a dynamical Fourier boundary condition driven by a noise. A boundary condition is called dynamical if it involves the time derivative. 

We consider $D$, an open bounded Lipschitz  domain of $\R^n$ with boundary $\partial D$, and let $Y=[0, 1[ ^n$ the representative cell. Denote by $O$ an open subset of $Y$ with a smooth boundary $\partial O$, such that $\o{O}\subset Y$, and let $Y^*=Y\setminus \o{O}$. The elementary cell $Y$ and the small cavity or hole $O$ inside it are used to model small scale obstacles or heterogeneities in a physical medium $D$. Denote by $O^{\e,k}$ the translation of $\e O$ by $\e k$, $k\in\Z^{n}$. We make the assumption that the holes do not intersect the boundary $\partial D$ and we denote by $\mathcal{K^\e}$ the set of all $k\in\Z^n$ such that the cell $\e k +\e Y$ is strictly included in $D$. The set of all such holes will be denoted by $O^{\e}$, i.e.
$$O^{\e}:=\bigcup_{k\in\mathcal{K^\e}}O^{\e,k}=\bigcup_{k\in\mathcal{K^\e}}\e(k+O),$$
and set
$$D^\e:=D- \o{O^{\e}}.$$
By this construction, $D^\e$ is a periodically perforated domain with holes of size of the same order as the period.
One of the difficulties of the homogenization in perforated domains consists in the fact that the $\e$- problems are defined in different domains. In \cite{T80}, \cite{CDE96}, \cite{A92}, \cite{A91} suitable extensions for the velocity and for the pressure were defined to overcome this difficulty. 

The evolution Stokes equation in the domain $D^\e$ with a stochastic dynamical boundary condition on the boundaries of the holes is given by

\begin{equation}
\label{system}
\left\{
\begin{array}{rll}
d u^\e(t,x) &=\left [ \nu\Delta u^\e (t,x)- \nabla p^\e(t,x)+f(t,x) \right ]dt+g_1(t) d W_1(t) &\mbox{in}\  D^\e, \\
\mbox{div } u^\e(t,x) &=0&\mbox{in}\  D^\e, \\
\e^2 d u^\e(t,x) &=-\left[\nu \dfrac{\partial u^\e(t,x)}{\partial n} -p^\e(t,x) n + b \e u^\e(t,x)\right] dt+\e g_2^\e(t) d W_2(t)&\mbox{on}\  \partial O^\e, \\
u^\e(0,x)&=u_0^\e(x)&\mbox{in}\  D^\e, \\
u^\e(0,x)&=v_0^\e (x)&\mbox{on}\  \partial O^\e, \\
\end{array}
\right.
\end{equation}
where $u^\e$ is the velocity of the fluid and $p^\e$ is the pressure.

Using It\^o's formula and stochastic calculus we are able to prove some uniform estimates in some functional spaces that are $\e$- dependent. Our particular boundary condition makes it difficult to extend the velocity continuously in $H^1(D)^n$, so we chose to use the trivial extension by $0$ of the velocity as well as of the gradient of the velocity. This extension is a continuous one in $L^2 (D)$ for which the uniform estimates still hold. Our extension is not divergence free, so we cannot expect the homogenized solution to be divergence free. Hence we cannot use test functions that are divergence free in the variational formulation, which implies that the pressure has to be included. Because of the low regularity in time of the stochastic process, we have to define a more regular in time pressure $P(t) = \int_0^t p(s) ds$ in Theorem \ref{Pe}. We apply to it the same extension and we will recover in the limit the information into the coefficients of the homogenized equation.

For the convergence of the boundary integrals we use an idea from \cite{CDE96} to rewrite an integral on the boundary in terms of an integral on the whole domain. In this way, the integrals over the boundary $\partial O^\e$ become in the limit integrals over the whole domain $D$. In particular, in the case of the stochastic integral on the boundaries of the holes, we combined the idea with the use of the decomposition of the Wiener process in terms of the basis, and that it is of trace class.

Our process depends on three variables: $\omega,\ t,\ \mbox{and } x$ but the oscillations appear only in $x$, hence we will use the two scale convergence method in the space variable but with a parameter $(\omega, t) \in \Omega \times [0,T]$. One of the condition to be able to use two scale convergence is the uniform boundedness with respect to $\e > 0$ of the $L^2$- norm. We were not able to show uniform bounds for $\omega \in \Omega$ or for any $t\in [0,T]$, but only in $L^2(\Omega\times [0,T]\times D)^n$, thus we decided using  this type of convergence. Most of the results concerning the two scale convergence are being extended in a straighforward way to the results of our paper. We gather the results we use in Section \ref{section4}. The convergence in two scale is a stronger type of convergence than the weak convergence (see Corollary \ref{2sc-w}), and so the convergence to the homogenized solution will be the weak one in $L^2(\Omega\times [0,T]\times D)^n$. This convergence although weak in the deterministic sense,  is strong in the probability sense.

The paper will be organized as follows: in Section \ref{section2} we formulate more precisely our problem, set the functional setting and give the assumptions needed for the  rest of the paper. In Section \ref{section3} we study the problem \eqref{system} in the perforated domain and show the existence and uniqueness in Theorem \ref{mild}. The estimates, needed for the two scale convergence method,  are derived. We also introduce the pressure in Theorem \ref{mild} and setup the variational formulation to be used for the passage to the limit. In Section \ref{section4}, we introduce the two scale convergence and give a number of results that we will use. In Section \ref{section5}, we pass to the limit in the variational formulation \eqref{varforPt}. We derive the cell problem \eqref{cellsystem} and the equation for the homogenized solution $u^*$ in \eqref{eqhom'}. 
\section{Preliminaries and assumptions}
\label{section2}

In this section we introduce some functional spaces and assumptions in order to study the problem 
\eqref{system}.

\subsection{Functional setting }

Let us introduce the following Hilbert spaces

\begin{equation}
\label{spaceL2}
\mf{L}^2_\e:=L^2(D^\e)^n\times L^2(\partial O^\e)^n,
\end{equation}

\begin{equation}
\label{spaceHe}
\mf{H}^1_\e:=H^1(D^\e)^n\times H^{\frac{1}{2}}(\partial O^\e)^n.
\end{equation}

equipped respectively with the inner products

\[
  \langle \mf{U},\mf{V}\rangle=\int_{D^\e} \big[  u(x) \cdot v(x)\big] dx 
  +\int_{\partial O^{\e}} \big[  \o{u}(x') \cdot \o{v}(x')\big] d\sigma(x'),
\]
and 
\[
 ((\mf{U},\mf{V})) = \int_{D^\e} \big[ \nabla u(x) \cdot\nabla v(x)\big] dx.\]

We introduce the bounded linear and surjective operator
$\gamma^\e:H^1(D^\e)^n\mapsto H^{\frac{1}{2}}(\partial O^\e)^n$ such that
$\gamma^{\e}u=u|_{\partial O^\e}$  for all $u\in C^\infty\left(\overline{D^\e}\right)^n$. $\gamma^\e$ is the  trace operator, (see \cite{sohr}, pp47). We denote by $H^{-\frac{1}{2}}(\partial O^\e)^n$ the dual space of 
$H^{\frac{1}{2}}(\partial O^\e)^n$.

We denote by $\mf{H}^\e$ the closure of $\mathcal{V}^\e$ in $\mf{L}^2_\e$, and by $\mf{V}^\e$ the closure of $\mathcal{V}^\e$ in $\mf{H}^1_\e$,  
where 

\begin{equation}
\label{spacemcV}
\mathcal{V}^\e:=\left\{ \mf{U}=\left(\begin{array}{c} 
u\\
\o{u}
\end{array}\right)
\in C^\infty\left(\overline{D^\e}\right)^n\times \gamma^\e(C^\infty\left(\overline{D^\e}\right)^n)\ \ |\ \ \mbox{div }u=0,\ \o{u}=\e u_{\partial O^\e},\ u=0 \mbox{ on } \partial D\right\},
\end{equation}

Let us also denote by $H^\e$ and $V^\e$ the following functional spaces

\begin{equation*}
H^\e:=\left\{u \in L^2(D^\e)^n\ |\ \mbox{div }u=0\mbox{ in } D^\e,\ u\cdot n=0 \mbox{ on } \partial D,\ u\cdot n=\e \o{u}\cdot n\mbox{ on } \partial O^\e\mbox{ for }\o{u}\in L^2(\partial O^\e)^n\right\},
\end{equation*}
\begin{equation*}
V^\e:={\left\{u \in H^1(D^\e)^n\ |\ \mbox{div }u=0\mbox{ in } D^\e,\ \ u=0 \mbox{ on } \partial D\right\}}.
\end{equation*}
Let $\Pi^\e: \mf{L}^2_\e \mapsto L^2(D^\e)^n$ be the operator that represents the projection onto the first component, i.e. $\Pi^\e \mf{U} =u$, for every $\mf{U}=(u,\o{u})\in \mathbf{L}^2_\e$,
and let $B^\e: V^\e \mapsto H^{-\frac{1}{2}}(\partial O^\e)^n$ denote the linear operator defined by $B^\e u=\dfrac{\partial u}{\partial n}$.

$\mf{H}^\e$ and $\mf{V}^\e$ are separable Hilbert spaces with the inner products and norms inherited from $\mf{L}^2_\e$ and $\mf{H}^1_\e$ respectively:
\[
\|\mf{U}\|^2_{\mf{H}^\e}=\langle \mf{U},\mf{U}\rangle\, , %\mbox{ \rm and }
%  \langle \mf{U},\mf{V}\rangle=\int_{D^\e} \big[  u(x) \cdot v(x)\big] dx 
%  +\int_{\partial O^{\e}} \big[  \o{u}(x') \cdot \o{v}(x')\big] d\sigma(x'),
\]
\[
 \|\mf{U}\|_{\mf{V}^\e}^2 = ((\mf{U},\mf{U})) \, , %\mbox{ \rm and }
% ((\mf{U},\mf{V})) = \int_{D^\e} \big[ \nabla u(x) \cdot\nabla v(x)\big] dx,
\]
and $H^\e$ and $V^\e$ are also separable Hilbert spaces with the norms induced by the projection $\Pi^\e$.

Denoting by $(\mf{H}^\e )'$ and $(\mf{V}^\e )'$ the dual spaces, if we identify  $\mf{H}^\e$ with  $(\mf{H}^\e) '$
then we have the Gelfand triple $\mf{V}^\e\subset \mf{H}^\e \subset (\mf{V}^\e)'$ with continuous injections. 

We denote the dual pairing between $\mf{U}\in \mf{V}^\e$ and $\mf{V}\in (\mf{V}^\e)'$ by
$\langle  \mf{U},\mf{V}\rangle_{\langle\mf{V}^\e,(\mf{V}^\e)'\rangle}$.
When $\mf{U}, \mf{V}\in \mf{H}^\e$, we have $\langle \mf{U},\mf{V}\rangle_{\langle\mf{V}^\e,(\mf{V}^\e)'\rangle}=\langle \mf{U},\mf{V}\rangle$.

Assume that $b$ is a strictely positive constant, and define the linear operator $\mf{A}^\e : D(\mf{A}^\e) \subset \mf{H}^\e\mapsto (\mf{H}^\e)'$:

\begin{equation}
\label{defAe}
\mf{A}^\e \mf{U}=\mf{A}^\e \left( \begin{array}{c}
u\\
\o{u}
\end{array}\right)
=
\left(\begin{array}{c}
-\nu \Delta u\\
\dfrac{\nu}{\e} B^\e u+\dfrac{b}{\e}\o{u}
\end{array}\right),
\end{equation}
with
$$D(\mf{A}^\e)=\{ \mf{U}\in \mf{V}^\e  | -\Delta u \in L^2(D^\e)^n\ \mbox{ and } \dfrac{\partial u}{\partial n} \in L^2(\partial O^\e)^n\},$$
and
$$\mf{A}^\e \mf{U} \cdot \mf{V}=\displaystyle\int_{D^\e} \nu \nabla u \nabla v dx+\int_{\partial O^\e} \e b \gamma^\e(u)\gamma^\e (v) d \sigma.$$

Let  $\mathbf{F}^\e \in L^2(0,T;L^2(D^\e)^n\times L^2(\partial O^\e)^n)$ be defined by
\begin{equation}
\label{defFe}
\mathbf{F}^\e(t,x)=\left(\begin{array}{c}
f(t,x)\\
0
\end{array}\right),
\end{equation}

%Given $p>1$ and $\gamma\in(0,1)$ and a Hilbert space $X$
%let $W^{\gamma,p}(0,T;X)$ be the Sobolev space of all $u\in L^{p}(0,T;X)$
%such that
%$$\int_{0}^{T}\int_{0}^{T}\frac{\|u(t)-u(s)\|_{X}^{p}}{|t-s|^{1+\gamma p}}dtds
%<\infty$$
%endowed with the norm
%$$\| u\|^{p}_{W^{\gamma,p}(0,T;X)} = \int_{0}^{T}\|u(t)\|_{X}^{p}dt
%+ \int_{0}^{T}\int_{0}^{T}\frac{\|u(t)-u(s)\|_{X}^{p}}{|t-s|^{1+\gamma p}}dtds.$$
%Let $K$ be another separable Hilbert space. Let us denote by  $L_{2}(K,H)$ the set of 
%Hilbert-Schmidt operators from $K$ to $H$.

Let $Q_1$ and $Q_2$ be linear positive operators in  
$L^2(D)^n$ of trace class. Let $(W_1(t))_{t\geq 0}$ and $(W_2(t))_{t\geq 0}$ be two mutually independent 
$L^2(D)^n$- valued Wiener processes defined on the complete probability space $(\Omega,\mathcal{F}, \mathbb{P})$ endowed with the canonical filtration $(\mathcal{F}_t)_{t\geq 0}$ and with the covariances $Q_1$ and $Q_2$. The expectation is denoted by $\mathbb{E}$. If $K$ and $H$ are two separable Hilbert spaces, then we will denote by $L_2(K,H)$ the space of bounded linear operators that are Hilbert-Schmidt from $K$ in $H$. If $Q$ is a linear positive operator in $K$ of trace class, then we will denote by $L_Q(K,H)$, the space of bounded linear operators that are Hilbert-Schmidt from $Q^{\frac{1}{2}} K$ to $H$, and the norm will be denoted by $\|\cdot\|_Q$.

Let us denote by

\begin{equation}
\label{defGe}
\mathbf{G}^\e(t)=\left(\begin{array}{cc}
g_1(t) & 0\\
0 & \e g_{2}^\e(t)
\end{array}\right), \qquad  \mathbf{W}(t)=(W_{1}(t), W_{2}(t)).
\end{equation}

The system \eqref{system} can be rewritten:

\begin{equation}
\label{system1}
\left\{
\begin{array}{rll}
d \mf{U}^\e(t) &+\mathbf{A}^\e \mf{U}^\e (t) dt=\mathbf{F}(t) dt+\mathbf{G}^\e(t) d \mathbf{W}, \\

\mf{U}^\e(0)&=\mf{U}_0^\e=\left( \begin{array}{c}
u_0^\e\\
\overline{u_{0}^{\e}}=\e v_0^\e
\end{array}
\right).
\end{array}
\right.
\end{equation}

\begin{lemma} (Properties of the operator $\mathbf{A}^\e$) For every $\e > 0$, the linear operator $\mf{A}^\e$ is positive and self-adjoint in $\mf{H}^\e$.
\end{lemma}
\begin{proof}
The operator is obviously symmetric, since for every $\mf{U},\ \mf{V} \in \mf{H}^\e$:
$$\langle \mf{A}^\e \mf{U}, \mf{V}\rangle =\nu \displaystyle\int_{D^\e} \nabla u \nabla v dx+\int_{\partial O^\e} \e b \gamma^\e(u)\gamma^\e (v) d \sigma$$
and also coercive:
$$\langle \mf{A}^\e \mf{U}, \mf{U}\rangle = \nu \displaystyle\int_{D^\e} \nabla u \nabla u dx+\int_{\partial O^\e} \e b \gamma^\e(u)\gamma^\e (u) d \sigma \geq\nu ||\mf{U}||^2_{\mf{V}^\e}.$$

To show that it is self-adjoint it will be enough to show that 
$$D(\mf{A}^\e) = \{\mf{U} \in \mf{V}^\e\ |\  \langle \mf{A}^\e \mf{U} ,\mf{V} \rangle \leq C ||\mf{V}||_{\mf{H}^\e} \mbox{ for every } \mf{V} \in \mf{V}^\e\}.$$

But, $\langle \mf{A}^\e \mf{U}, \mf{V}\rangle \leq C ||\mf{V}||_{\mf{H}^\e}$ for every $\mf{V} \in \mf{V}^\e$ is equivalent to
\begin{equation} 
\begin{split}
\nu \displaystyle\int_{D^\e} \nabla u \nabla v dx \leq C||v||_{L^2(D^\e)^n} + C ||v||_{L^2(\partial O^\e)^n}\ \forall v\in V^\e \Longleftrightarrow\\
\nu \displaystyle\int_{D^\e} -\Delta u  v dx + \displaystyle \int_{\partial O^\e} \dfrac{\partial u}{\partial n} v d\sigma \leq C||v||_{L^2(D^\e)^n} + C ||v||_{L^2(\partial O^\e)^n}\ \forall v\in V^\e \Longleftrightarrow\\
\end{split}
\end{equation}
$\Delta u \in L^2(D^\e)^n$ and $\dfrac{\partial u}{\partial n} \in L^2(\partial O^\e)^n$, so $\mf{U} \in D(\mf{A}^\e)$.

Now we use Proposition A.10, page 389 from \cite{DPZ} to infer that $\mf{A}^\e$ is self-adjoint.
\end{proof}

For $\alpha>0$, let us denote by  $(\mf{A}^\e)^{\alpha}$ the $\alpha$-power of the operator $\mf{A}^\e$
and $D((\mf{A}^\e)^{\alpha})$ its domain. In particular,
$$D((\mf{A}^\e)^{0})=\mf{H}^\e,\quad {\rm and} \quad D((\mf{A}^\e)^{1/2})=\mf{V}^\e.$$

Indeed, for every $\e > 0$ and using the trace inequality
$$\|\gamma_{\e}(u)\|^{2}_{L^2(\partial O^{\e})^n}\leq\|\gamma_{\e}(u)\|^{2}_{H^{\frac{1}{2}}(\partial O^{\e})^n}
\leq C(\e)\|u\|^{2}_{H^1(D^\e)^n},$$
it follows
\begin{align*}
\langle \mf{A}^\e \mf{U}, \mf{U}\rangle &=\nu \displaystyle\int_{D^\e} \nabla u \nabla u dx+\int_{\partial O^\e} \e b \gamma^\e(u)\gamma^\e (u) d \sigma \leq C \|\nabla u\|^{2}_{L^2(D^\e)^{n\times n}}.
\end{align*}

Denote  by $S_{\e}(t)$  the analytic semigroup generated by $\mf{A}^{\e}$ (see \cite{pazy}). 
%By the 
%general theory of analytic semigroups or by explicit computation based on the spectral
%representation, for every $\alpha>0$ we have $\left|
%(\mf{A}^\e)^{\alpha}S_{\e}(t)\right| _{\mf{H}^\e}\leq\frac{C_{\alpha}}{t^{\alpha}}$ for some constant $C_%{\alpha}>0$. 

\begin{lemma}
\label{traceineqe}
Let $\phi \in H^1(D^\e)$. Then
$$\sqrt{\e} \| \phi\|_{L^2(\partial O^\e)} \leq C \|\phi \|_{H^1(D^\e)}.$$
\end{lemma}

\begin{proof}
For each $k\in \mathcal{K}^\e$, let $\phi^{\e, k}$ be the function defined in $Y^*$ by $\phi^{\e, k}(x) = \phi\left(  \frac{x}{\e} -k \right)$. Then,  we have
$$\| \phi^{\e, k} \|_{L^2(\partial O)} \leq C \left( \phi^{\e, k} \|_{L^2(Y^*)} + \| \nabla \phi^{\e, k} \|_{L^2(Y^*)^n} \right).$$
We use the change of variables $x=\e x +\e k$ and add over all $k\in\mathcal{K}^\e$ to obtain the result.
\end{proof}

In the next section, we will set up the assumptions.
\subsection{Assumptions} We assume that the operator 

$$g_1\in C([0,T]; L_{Q_{1}}(L^2(D)^n, L^2(D)^n)$$
and let

$$g_{21}\in C([0,T]; L_{Q_{2}}(L^2(D)^n,H^1(D)^n)$$ 
and
$$g_{22}\in C([0,T]; L_{Q_{2}}(L^2(D)^n,L^2(\partial O)^n)$$ 
such that there exists a positive constant $C_T$ 

\begin{equation}
\label{eqCT}
\begin{split}
||g_1(t)||^2_{Q_1} := 
\sum_{j=1}^\infty \lambda_{j1}||g_1(t) e_{j1}||^2_{L^2(D)^n} \leq C_T, \ t\in [0,T],\\
||g_{21}(t)||^2_{Q_2} := 
\sum_{j=1}^\infty \lambda_{j2}||g_{21}(t) e_{j2}||^2_{H^1(D)^n} \leq C_T, \ t\in [0,T],\\
||g_{22}(t)||^2_{Q_2} := 
\sum_{j=1}^\infty \lambda_{j2} ||g_{22}(t) e_{j2}||^2_{L^2(\partial O)^n} \leq C_T, \ t\in [0,T],\\
\end{split}
\end{equation}
where $\{e_{j1}\}_{j=1}^\infty$ and $\{e_{j2}\}_{j=1}^\infty$ are respectively the eigenvalues for $Q_1$ and $Q_2$, and $\{\lambda_{j1}\}_{j=1}^\infty$ and $\{\lambda_{j2}\}_{j=1}^\infty$ are the corresponding sequences of eigenvalues. For any element $h \in L^2(\partial O)$, we define the element $\mathcal{R}^\e h \in L^2(\partial O^\e)$ by 
\begin{equation}
\label{defRe}
\mathcal{R}^\e h(x) = h\left(\dfrac{x}{\e}\right)
\end{equation}
where $h$ is considered $Y-$ periodic.
\begin{lemma}
\label{Reh}
There exists a constant $C$ independent of $\e >0$ and $h \in L^2(\partial O)^n)$ such that
$$\sqrt{\e}\|\mathcal{R}^\e h\|_{L^2(\partial O^\e)^n} \leq C \|h\|_{L^2(\partial O)^n}.$$
\end{lemma}
\begin{proof}
We use the definition of $\mathcal{R}^\e h$ to obtain after a change of variables:
\begin{equation*}
\begin{split}
\|\mathcal{R}^\e h\|_{L^2(\partial O^\e)^n}^2 &= \int_{\partial O^\e} h^2\left( \frac{x'}{\e} \right) d\sigma(x') =  \e^{n-1}\sum_{k \in \mathcal{K}^\e} \|h\|_{L^2(\partial O)^n}^2\\
 &\leq \e^{n-1} C \e^{-n}  \|h\|_{L^2(\partial O)^n}^2 = C \e^{-1}\|h\|_{L^2(\partial O^\e)^n}^2,
\end{split}
\end{equation*}
which gives the result.
\end{proof}

We define $g_2^\e(t)$ by
$$g_2^\e(t)=g_{21}(t)+\mathcal{R}^\e g_{22}(t),$$
where $\mathcal{R}^\e g_{22} \in C([0,T]; L_{Q_{2}}(L^2(D)^n,L^2(\partial O ^\e )^n)$.
Moreover, throughout the paper we will assume that 
$\mf{U}_{0}^{\e}= (u^{\e}_{0},\overline{u^{\e}_{0}}=\e v_0^\e)$ is an $\mathcal{F}_{0}-$ measurable $\mf{H}^\e-$ valued random variable and there exists a constant $C$ independent of $\e$, such that for every $\e > 0$:

\begin{equation}\label{u0}
\mathbb{E}\|u^{\e}_{0}\|_{L^2(D^\e)^n}^2+\e\mathbb{E}\|v_0^\e\|_{L^{2}(\partial O^\e)}^2\leq C.
\end{equation}

\section{The microscopic model}
\label{section3}
In this section, we will state the well posedness of system \eqref{system} and some uniform estimates wrt $\e$.

\begin{theorem}\label{mild}
 (Well posedness of the microscopic model)
Assume that \eqref{eqCT} holds, then for any $T>0$ and any $\mf{U}^\e_{0}$ an $\mf{H}^\e-$ valued measurable random variable, the system \eqref{system} has a unique mild solution 
$\mf{U}^\e\in L^{2}(\Omega, C([0,T], \mf{H}^\e)\cap L^{2}(0,T; \mf{V}^\e))$, 

\begin{equation}\label{mildeq} 
\mf{U}^\e(t)=S_{\e}(t)\mf{U}^\e_{0}+\int_{0}^{t}S_{\e}(t-s)\mathbf{F}(s)ds
+\int_{0}^{t}S_{\e}(t-s)\mathbf{G}^\e(s)d\mathbf{W}, \quad t\in [0,T].
\end{equation}

The mild solution $\mf{U}^\e$ is also a weak solution, that is,  $\mathbb{P}$-a.s.

\begin{equation}\label{var}
\langle \mf{U}^\e(t), \bm{\phi}\rangle+\int_{0}^{t}\langle (\mf{A}^\e)^{1/2} \mf{U}^\e(s),(\mf{A}^\e)^{1/2}\bm{\phi}\rangle ds
=\langle \mf{U}^\e_{0},\bm{\phi}
\rangle +\int_{0}^{t}\langle \mf{F}(s),\bm{\phi}\rangle ds+\int_{0}^{t}\langle \mf{G}^\e(s)d\mf{W}(s),\bm{\phi}\rangle
\end{equation}
for $t\in [0,T]$ and $\bm{\phi}\in \mf{V}^\e$. 

Moreover, if \eqref{u0} holds then for every $\e>0$
%$\sup_{\e > 0}\mathbb{E}\|\mf{U}^\e_{0}\|_{\mf{H}^\e}^{2} < \infty$, and for every $\e>0$

\begin{equation}\label{U-average1}
\mathbb{E}\|\mf{U}^\e(t)\|_{\mf{H}^\e}^{2}
+\mathbb{E}\int_{0}^{t}\|\mf{U}^\e(s)\|_{\mf{V}^\e}^{2}ds
\leq C_{T}(1+\mathbb{E}\|\mf{U}^\e_{0}\|_{\mf{H}^\e}^{2}), \quad t\in [0,T]
\end{equation}
and 

\begin{equation}\label{U-average2}
\mathbb{E}\sup_{t\in[0,T]}\|\mf{U}^\e(t)\|_{\mf{H}^\e}^{2}
\leq C_{T}(1+\mathbb{E}\|\mf{U}^\e_{0}\|_{\mf{H}^\e}^{2}).
\end{equation}

\end{theorem}

\begin{proof}
Since the operator $\mf{A}^{\e}$ is the generator of a strongly continuous semigroup 
$S_{\e}(t),\ t\geq 0$ in $\mf{H}^\e$ and using the assumption \eqref{eqCT}, then the existence and uniqueness of mild solutions in  $\mf{H}^\e$ is a consequence of Theorem 7.4 of \cite{DPZ}.
The regularity in  $\mf{V}^\e$ is a consequence of the estimates below.

Applying the It\^{o} formula to $\|\mf{U}^\e(t)\|_{\mf{H}^\e}^{2}$, we get that
\begin{align*}
d\|\mf{U}^\e(t)\|_{\mf{H}^\e}^{2}&=2\langle \mf{U}^\e(t),d\mf{U}^\e(t)\rangle dt+\| \mf{G}^\e(t)\|_{Q}^{2}dt\\
&=-2\langle \mf{A}^\e \mf{U}^\e(t),\mf{U}^\e(t)\rangle dt
+2 \langle \mf{F}(s),\mf{U}^\e(t)\rangle dt+2\langle \mf{G}^\e(t)d\mf{W},\mf{U}^\e(t)\rangle dt
+\| \mf{G}^\e(t)\|_{Q}^{2}dt.
\end{align*}
Hence,

\begin{align}\label{1}
\|\mf{U}^\e(t)\|_{\mf{H}^\e}^{2}&+2\int_{0}^{t}\|(\mf{A}^\e)^{1/2} \mf{U}^\e(s)\|_{\mf{H}^\e}^{2}ds\leq 
\|\mf{U}^\e_{0}\|_{\mf{H}^\e}^{2}+\int_{0}^{t}\|\mf{F}(s)\|_{\mf{H}^\e}^{2}ds+\int_{0}^{t}\|\mf{U}^\e(s)\|_{\mf{H}^\e}^{2}ds
\nonumber\\
&+2\int_{0}^{t} \langle \mf{G}^\e(s)d\mf{W},\mf{U}^\e(t)\rangle +\int_{0}^{t}\| \mf{G}^\e(s)\|_{Q}^{2}ds.
\end{align}
Taking the expected value yields

\begin{equation}\label{2}
\mathbb{E}\|\mf{U}^\e(t)\|_{\mf{H}^\e}^{2}\leq 
\mathbb{E}\|\mf{U}^\e_{0}\|_{\mf{H}^\e}^{2}+\int_{0}^{t}\|\mf{F}(s)\|_{\mf{H}^\e}^{2}ds+
\int_{0}^{t}\mathbb{E}\|\mf{U}^\e(s)\|_{\mf{H}^\e}^{2}ds
+\int_{0}^{t}\| \mf{G}^\e(s)\|_{Q}^{2}ds
\end{equation}
and 
\begin{equation}\label{3}
\mathbb{E}\int_{0}^{t}\|\mf{U}^\e(s)\|_{\mf{V}^\e}^{2}ds
\leq \mathbb{E}\|\mf{U}^\e_{0}\|_{\mf{H}^\e}^{2}
+\int_{0}^{t}\|\mf{F}(s)\|_{\mf{H}^\e}^{2}ds+\int_{0}^{t}\mathbb{E}\|\mf{U}^\e(s)\|_{\mf{H}^\e}^{2}ds 
+\int_{0}^{t}\| \mf{G}^\e(s)\|_{Q}^{2}ds
\end{equation}

Now  \eqref{U-average1} follows from using Gronwall's lemma in \eqref{2}.

On the other side, \eqref{1} implies that
\begin{align*}
\sup_{0\leq t\leq T}\|\mf{U}^\e(t)\|_{\mf{H}^\e}^{2}&+
2\int_{0}^{T}\|\mf{U}^\e(s)\|_{\mf{V}^\e}^{2}ds\leq 
\|\mf{U}^\e_{0}\|_{\mf{H}^\e}^{2}+\int_{0}^{T}\|\mf{F}(s)\|_{\mf{H}^\e}^{2}ds+\int_{0}^{T}\|\mf{U}^\e(s)\|_{\mf{H}^\e}^{2}ds\\
&+2\sup_{0\leq t\leq T}\left |\int_{0}^{t} \langle \mf{G}^\e(s)d\mf{W},\mf{U}^\e(t)\rangle \right |
+\int_{0}^{T}\| \mf{G}^\e(s)\|_{Q}^{2}ds
\end{align*}
Moreover using the Burkholder-David-Gundy inequality and the Young inequality, we get that

\begin{align*}
\mathbb{E}\sup_{0\leq t\leq T}\left |\int_{0}^{t} \langle \mf{G}^\e(s)d\mf{W},\mf{U}^\e(t)\rangle \right | &\leq 
\mathbb{E}\left(\int_{0}^{T}\left |\mf{G}^\e(s) \mf{U}^\e(s)\right |^{2}ds\right)^{1/2}\\
&\leq \frac{1}{2} \mathbb{E}\sup_{0\leq t\leq T}\|\mf{U}^\e(t)\|_{\mf{H}^\e}^{2}
+ C\int_{0}^{T}|\mf{G}^\e(s)|^{2}ds\\
&\leq \frac{1}{2} \mathbb{E}\sup_{0\leq t\leq T}\|\mf{U}^\e(t)\|_{\mf{H}^\e}^{2}+C_{T}
\end{align*}

Now, plugging this estimate in the previous one and using Gronwall's lemma completes the proof of 
\eqref{U-average2}.
\end{proof}

\begin{theorem}\label{uniform}
Assume that the assumptions of Theorem \ref{mild} hold. Then, for every $t\in [0,T]$ and every $\e>0$

\begin{align}\label{uniform1}
\mathbb{E}\|u^{\e}(t)\|_{H^\e}^{2} &
+\e^{2}\mathbb{E}\|\gamma_{\e}u^{\e}(t)\|_{L^{2}(\partial O^{\e})^n}^{2} 
+\mathbb{E}\int_{0}^{t}\left( \|u^{\e}(s)\|_{V^\e}^{2}
+\e^{2}\|\gamma_{\e} u^{\e}(s)\|_{H^{\frac{1}{2}}(\partial O^{\e})^n}^{2}\right)ds\nonumber\\
&\leq C_{T}(1+\mathbb{E}\|\mf{U}^\e_{0}\|_{\mf{H}^\e}^{2}), 
\end{align}
and

\begin{equation}\label{uniform2}
\mathbb{E}\sup_{t\in[0,T]}\left(\|u^{\e}(t)\|_{H^\e}^{2}
+ \e^{2}\|\gamma_{\e}u^{\e}(t)\|_{L^{2}(\partial O^{\e})^n}^{2}\right)
 \leq C_{T}(1+\mathbb{E}\|\mf{U}^\e_{0}\|_{\mf{H}^\e}^{2}).
\end{equation}
Moreover, for $q\geq 2$ 

\begin{align}\label{uq}
\sup_{t\in[0,T]}\left(\mathbb{E}\|u^{\e}(t)\|_{H^\e}^{q} 
+\e^{2}\mathbb{E}\|\gamma_{\e}u^{\e}(t)\|_{L^{2}(\partial O^{\e})^n}^{q} \right)
\leq C(q, T)(1+\mathbb{E}\|\mf{U}^\e_{0}\|_{\mf{H}^\e}^{q}), 
\end{align}

\end{theorem}

\begin{proof}
\eqref{uniform1}, \eqref{uniform2} and \eqref{uq} are consequences of Theorem \ref{mild}. 
\end{proof}

For any function $z^\e$ defined in $D^\e$ let us denote by $\wt{z}^{\e}$ the extension of $z^{\e}$ to $D$ by
\begin{equation}
\wt{z}^{\e}:=\left\{\begin{array}{lr}
z^{\e} &{\rm on}\ D^\e\\
0 & {\rm on} \ D\setminus D^\e
\end{array}
\right.
\end{equation}

\begin{theorem}\label{Pe}
Let us define $U^\e(t)=\int_0^t u^\e(s)ds$, where $u^\e$ is given by the previous theorem, with $U^\e\in C([0,T]; V^\e)$. Then, there exists a unique vector valued random variable $P^\e: \Omega \mapsto C([0,T]; L^2(D^\e))$, such that for every $t\in [0,T]$ we have:

\begin{equation}\label{varforPt}
\begin{split}
\int_{D^\e} \left(u^\e(t)-u^\e_{0}-\int_{0}^{t}f(s)-\int_{0}^{t}g_{1}(s)dW_{1}(s)\right) \phi dx
+\int_{D^\e}\left(\nu \nabla U^\e(t) \nabla\phi-P^\e(t) \operatorname{div}\phi\right) dx\\
=\int_{\partial O^\e} \left(\e^2 u^\e(t)-\e^2 v^\e_0 +\int_{0}^{t} \e g_{2}^\e(s)dW_{2}(s) ds
-\e b U^\e(t)\right) \phi d\sigma ,
\end{split}
\end{equation}
$\mathbb{P}$-a.s., and for every $\phi \in H_0^1(D)^n$.

Moreover,
\begin{equation}
\label{estPet}
\sup_{0\leq t\leq T}\E\|P^\e(t)\|_{L^2(D^\e)^n}^{2}\leq C(T).
\end{equation}

\end{theorem}

\begin{proof}
We recall \eqref{var} which says that $\mf{U}^\e(t)=(u^\e, \e\gamma_{\e}(u^\e) )$ is a variational  solution in the following sense:

\begin{equation*}
\langle \mf{U}^\e(t), \bm{\phi}\rangle+\int_{0}^{t}\langle (\mf{A}^\e)^{1/2} \mf{U}^\e(s),
(\mf{A}^\e)^{1/2}\bm{\phi}\rangle ds
=\langle \mf{U}^\e_{0},\bm{\phi}\rangle 
+\int_{0}^{t}\langle \mf{F}(s),\bm{\phi}\rangle ds
+\int_{0}^{t}\langle \mf{G}^\e(s)d\mf{W}(s),\bm{\phi}\rangle
\end{equation*}
$\mathbb{P}-{\rm a.s.}$ for $t\in [0,T]$ and $\bm{\phi}\in \mf{V}^\e$.

In this variational formulation we consider the test function $\bm{\phi}=(\phi,0)$, where $\phi=P^{\e}\bm{\phi}\in V^\e$, hence 
$\mathbb{P}$-a.s. we get

\begin{equation}
\label{preP}
\langle u^\e(t)-u^{\e}_{0}-\nu\int_{0}^{t} \Delta u^\e(s)ds
-\int_{0}^{t}f(s)-\int_{0}^{t}g_{1}(s)dW_{1}(s),\phi\rangle=0
\end{equation}
$\forall t\in [0, T], \quad \forall \phi\in V^\e \cap H_0^1(D^\e)^n$.

Using Proposition I.1.1 and I.1.2. of \cite{temam}, we find for each $t\in [0,T]$,
 the existence of some $P^\e(t)\in L^{2}(D^\e)^n / \mathbb{R}$ such that $\mathbb{P}$-a.s.

\begin{equation}
\label{defPt}
u^\e(t)-u^\e_{0}-\nu \Delta U^\e(t)
-\int_{0}^{t}f(s)-\int_{0}^{t}g_{1}(s)dW_{1}(s)=-\nabla P^\e(t),
\end{equation}
where,  
$$L^{2}(D^\e)^n / \mathbb{R}=\left\{ P\in L^{2}(D^\e)^n, \quad \int_{D^{\e}}P(x)dx=0\right\}.$$

In particular, $\nabla P^\e\in C([0,T]; H^{-1}(D^\e)^n) \quad \mathbb{P}-{\rm a.s}.$.

The equation \eqref{defPt} may be written equivalently:
\begin{equation}
\label{defPt2}
\begin{split}
\int_{D^\e}\left(u^\e(t) - u^\e_{0}\right)\phi dx+\int_{D^\e}\left(\nu \nabla U^\e(t) \nabla\phi-
P^\e(t) \operatorname{div} \phi \right)dx\\
=\int_{D^\e}\int_{0}^{t}f(s)\phi dsdx+\int_{D^\e}\int_{0}^{t}g_{1}(s)\phi dW_{1}(s) dx,
\end{split}
\end{equation}
for every $\phi\in H^1_0(D^\e)^n$.
The equation \eqref{defPt} implies that $\nu\dfrac{\partial U^\e(t)}{\partial n}-P^\e(t) n 
\in H^{-\frac{1}{2}}(\partial O^\e)^n$, so for every $\phi \in H^1(D^\e)^n$, with $\phi_{|\partial D} =0$ we have
\begin{equation*}
\begin{split}
\left\langle u^\e(t)-u^\e_{0}-\int_{0}^{t}f(s)-\int_{0}^{t}g_{1}(s)dW_{1}(s), \phi \right\rangle
+\int_{D^\e}\left(\nu \nabla U^\e(t) \nabla\phi-P^\e(t) \operatorname{div}\phi\right) dx=\\
\left\langle \nu\dfrac{\partial U^\e(t)}{\partial n}-P^\e(t) n, \phi\right\rangle_{H^{-\frac{1}{2}}(\partial O^\e)^n,H^{\frac{1}{2}}(\partial O^\e)^n},
\end{split}
\end{equation*}
which implies, using \eqref{var} that
\begin{equation}
\label{4}
\begin{split}
-\left\langle \nu\dfrac{\partial U^\e(t)}{\partial n}-P^\e(t) n, 
\phi\right\rangle_{H^{-\frac{1}{2}}(\partial O^\e)^n,H^{\frac{1}{2}}(\partial O^\e)^n}\\
=\int_{\partial O^\e} \left(\e^2 u^\e(t) - \e \overline{u^\e_0} + \e b U^\e(t) 
- \int_0^t g_2^\e(s) d W_2(s)\right)\phi d\sigma,
\end{split}
\end{equation}
for every $\phi\in H^{\frac{1}{2}}(\partial O^\e)^n$ such that $\int_{\partial O^\e} \phi n d\sigma  =0$. So, there exists $c^\e\in C([0,T])$ such that if we re-denote $P^\e(t)=P^\e(t)+c^\e(t)$, \eqref{4} is satisfied for every $\phi \in H^{\frac{1}{2}}(\partial O^\e)^n$.  Using again \eqref{var},  we showed that there exists a unique $P^\e \in C([0,T],L^2(D^\e))$, such that for $t\in [0,T]$ \eqref{varforPt} holds.

To show \eqref{estPet} we consider the extension to the whole domain of the pressure $\wt{P}^\e(t)$, and compute $$\E\| \nabla \wt{P}^\e(t)\|_{H^{-1}(D)^n}^2 = \E \sup_{\overset{\phi \in H_0^1(D)^n}{\| \phi\|_{H_0^1(D)^n} \leq 1}} \left | \int_D  \wt{P}^\e(t) \operatorname{div} \phi dx \right |^2, $$ from \eqref{varforPt}. We show that it is bounded uniformly for $t \in [0,T]$, and then apply the result given by Proposition 1.2 from \cite{temam} to obtain \eqref{estPet}.
We estimate first $\left | \displaystyle\int_D  \wt{P}^\e(t) \operatorname{div} \phi dx \right |$ from \eqref{varforPt} and obtain by applying H{\"o}lder's inequality:
\begin{equation*}
\begin{split}
\left | \displaystyle\int_D  \wt{P}^\e(t) \operatorname{div} \phi dx \right | &\leq \|\phi\|_{H_0^1(D)^n} \left(\| u^\e(t)\|_{L^2(D^\e)^n} + \| u_0^\e\|_{L^2(D^\e)^n} + T \|f\|_{L^2([0,T]\times D)^n}\right) \\
&+\|\phi\|_{H_0^1(D)^n} \left(\| \int_0^t g_1(s) d W_1(s) \|_{L^2(D^\e)^n} +\nu \| U^\e(t) \|_{H_0^1(D^\e)^n} \right) \\
&+ \|\phi\|_{L^2(\partial O^\e)^n} \left( \|\e^2u^\e(t) + \e^2v_0^\e + \e b U^\e(t) + \int_0^t g_2^\e(s) dW_2(s) \|_{L^2(\partial O^\e)^n}\right).
\end{split}
\end{equation*}
We use Lemma \ref{traceineqe} and the estimates \eqref{uniform1} for $u^\e(t)$ and \eqref{u0} for the initial conditions to get
\begin{equation*}
\begin{split}
&\E\| \nabla \wt{P}^\e(t)\|_{H^{-1}(D)^n}^2 \leq C + C \E \|\int_0^t g_1(s) d W_1\|_{L^2(D^\e)^n}^2 + C \e  \E \|u^\e(t)\|_{L^2(\partial O^\e)^n}^2 +\\
& \e\E\|\int_0^t g_{21}(s) d W_2(s)\|_{L^2(\partial O^\e)^n}^2 +\e\E\|\int_0^t \mathcal{R}^\e g_{22}(s) d W_2(s)\|_{L^2(\partial O^\e)^n}^2.
\end{split}
\end{equation*}
Ito's isometry gives
\begin{equation*}
\begin{split}
&\E\| \nabla \wt{P}^\e(t)\|_{H^{-1}(D)^n}^2 \leq C + C\int_0^t \sum_{i=1}^\infty \lambda_{i1} \|g_1(s)e_{i1}\|_{L^2(D^\e)^n}^2ds + \\
& C\e\int_0^t \sum_{i=1}^\infty \lambda_{i2} \|g_{21}(s)e_{i2}\|_{L^2(\partial O^\e)^n}^2ds +C\e\int_0^t \sum_{i=1}^\infty \lambda_{i2} \|\mathcal{R}^\e g_{22}(s)e_{i2}\|_{L^2(\partial O^\e)^n}^2ds.
\end{split}
\end{equation*}
We use again Lemmas \ref{traceineqe} and \ref{Reh} and then property \eqref{eqCT} to obtain that $\E\| \nabla \wt{P}^\e(t)\|_{H^{-1}(D)^n}^2$ is bounded uniformly for $t \in [0,T]$.
\end{proof}

\section{Two scale convergence}
\label{section4}
We will summarize in this section several results about the two scale convergence that we will use throughout the paper. For the results stated without proofs, see \cite{A-2s} or \cite{JL13}. First we establish some notations of spaces of periodic functions. We denote by $C_{\#}^k (Y)$ the space of functions from $C^k(\o{Y})$, that have $Y-$ periodic boundary values. By $L^2_{\#}(Y)$ we understand the closure of $C_{\#}(Y)$ in $L^2(Y)$ and by $H^1_{\#}(Y)$ the closure of $C_{\#}^1(Y)$ in $H^1(Y)$.

\begin{definition}
\label{2scdef}
We say that a sequence $u^\e \in L^2(\Omega\times [0,T]\times D)^n$ two-scale converges to $u \in L^2(\Omega\times [0,T]\times D \times Y)^n$, and denote this convergence by
$$u^\e \overset{2-s}{\longrightarrow} u\ \ \ \mbox{ in } \Omega\times [0,T]\times D,$$
if for every $\Psi \in L^2(\Omega\times [0,T]\times D; C_{\#} (Y))^n$ we have
\begin{equation*}
\lim_{\e \to 0} \int_\Omega\int_0^T \int _D u^\e (\omega,t,x) \Psi(\omega,t,x,\frac{x}{\e}) dx dt d\mathbb{P} =\int_\Omega\int_0^T\int_D\int_Y u(\omega,t,x,y) \Psi(\omega,t,x,y) dy dx dt d\mathbb{P}.
\end{equation*}
\end{definition}

\begin{theorem}
\label{2scex}
Assume the sequence $u^\e$ is uniformly bounded in $L^2(\Omega\times [0,T]\times D)^n$. Then there exists a subsequence $(u^{\e'})_{\e' >0}$ and $u^0 \in L^2(\Omega\times [0,T]\times D \times Y)^n $ such that $u^{\e'}$ two-scale converges to $u^0$ in $L^2(\Omega\times [0,T] \times D)^n$.
\end{theorem}

\begin{cor}
\label{2sc-w}
Assume the sequence $u^\e\in L^2(\Omega\times [0,T]\times D)^n$, two-scale converges to $u^0 \in L^2(\Omega\times [0,T]\times D \times Y)^n$. Then, $u^\e$ converges weakly in $L^2(\Omega\times [0,T]\times D)^n$ to $\displaystyle\int_Y u^0(\omega,t,x,y) dy$.
\end{cor}

\begin{theorem}
\label{2scgrad}
Let for any $\e > 0$, $u^\e \in L^2(\Omega\times [0,T] ; H^1(D^\e)^n)$ such that $u^\e = 0$ on $\partial D$. Assume that $u^\e$ is a sequence uniformly bounded with respect to $\e > 0$, i.e. $$\sup_\e || u^\e ||_{L^2(\Omega\times [0,T] ; H^1(D^\e)^n)} < \infty.$$ Then, there exists $u^0 \in L^2(\Omega\times [0,T]; H_0^1(D)^n)$ and $u^1 \in L^2(\Omega\times [0,T]\times D ; H^1_{\#}(Y)^n)$, such that
$$\widetilde{u}^\e (\omega,t, x) \overset{2-s}{\longrightarrow} u^0 (\omega,t,x) \mathds{1}_{Y^*}(y)$$
and
$$\widetilde{\nabla u}^\e (\omega,t, x) \overset{2-s}{\longrightarrow} \left(\nabla_x u^0(\omega,t, x) + \nabla_y u^1 (\omega,t,x,y) \right)\mathds{1}_{Y^*}(y).$$
\end{theorem}

\begin{theorem}
\label{2scint}
Assume that the sequence $u^\e$ two scale converges to $u\in L^2(\Omega\times[0,T] \times D\times Y)^n$. Then the sequence $U^\e(\omega,t,x)$ defined by
$$U^\e(\omega,t,x) = \int_0^t u^\e(\omega, s, x) ds$$
two scale converges to $\displaystyle\int_0^t u(\omega, s, x,y) ds$.
\end{theorem}
\begin{proof}
We have for the sequence $u^\e$ that
\begin{equation*}
\lim_{\e \to 0} \int_\Omega\int_0^T \int _D u^\e (\omega,t,x) \Psi(\omega,t,x,\frac{x}{\e}) dx dt d\mathbb{P} =\int_\Omega\int_0^T\int_D\int_Y u(\omega,t,x,y) \Psi(\omega,t,x,y) dy dx dt d\mathbb{P},
\end{equation*}
for every $\Psi \in L^2(\Omega\times [0,T]\times D; C_{\#} (Y))^n$. Now if we choose $\Psi$ to be of the form
$$\Psi (\omega,s,x,y) = \mathds{1}_{[0,t]}(s)\Psi_1(\omega,x,y),$$
we obtain that
\begin{equation*}
\lim_{\e \to 0} \int_\Omega\int_0^t \int _D u^\e (\omega,s,x) \Psi_1(\omega,x,\frac{x}{\e}) dx d\mathbb{P} =\int_\Omega\int_0^t\int_D\int_Y u(\omega,s,x,y) \Psi_1(\omega,x,y) dy dx d\mathbb{P},
\end{equation*}
for any fixed $t\in[0,T]$ and any $\Psi_1 \in L^2(\Omega \times D; C_{\#}(Y))^n$. As a consequence, for $\Psi_2 \in L^2(\Omega\times [0,T]\times D; C_{\#} (Y))^n$ the sequence
$$h^\e(t) = \int_\Omega\int_0^t \int _D u^\e (\omega,s,x) \Psi_2(\omega,t,x,\frac{x}{\e}) dx ds d\mathbb{P} = \int_\Omega \int _D U^\e (\omega,t,x) \Psi_2(\omega,t,x,\frac{x}{\e}) dx d\mathbb{P}$$
is convergent for almost every $t\in[0,T]$ to
$$h(t) = \int_\Omega\int_0^t\int_D\int_Y u(\omega,s,x,y) \Psi_2(\omega,t,x,y) dy dx ds d\mathbb{P} = \int_\Omega\int_D\int_Y U(\omega,t,x,y) \Psi_2(\omega,t,x,y) dy dx d\mathbb{P}.$$
By applying H{\"o}lder's inequality,
$$h^\e(t) \leq ||u^\e||_{L^2(\Omega\times[0,T]\times D)^n} ||\Psi_2(t)||_{L^2(\Omega\times D \times Y)^n}.$$
We use the dominated convergence theorem to deduce that
$$\int_0^T h^\e(t) dt \to \int_0^T h(t) dt.$$
\end{proof}

\section{Homogenized equation}
\label{section5}
In order to find the homogenized equation, we first establish the two scale limits that we use in the variational formulation \eqref{varforPt} to pass to the limit when $\e\to 0$.
\subsection{Limit of the variational formulation \eqref{varforPt}}
\label{subsection51}

\

We notice that $\wt{u}_0^\e$ satisfies $\mathbb{E}\|\widetilde{u}^\e_{0}\|_{L^2(D^\e)^n}^2\leq C$ according to \eqref{u0}, so there exists $u_0 \in L^2(\Omega; L^2(D;L^2_{\#}(Y))^n$ such that 

\begin{equation}
\label{2s1}
\widetilde{u}_0^\e (x) \overset{2-s}{\longrightarrow} u_0 (x,y) \mathds{1}_{Y^*}(y)\ \ \ \mbox{ in } \Omega\times D.
\end{equation}

We recall that the processes $\wt{u}^\e$ and $\wt{\nabla u}^\e$ defined on the stochastic basis
$(\Omega, \mathcal{F}, (\mathcal{F}_{t}), \mathbb{P})$ are uniformly bounded in $L^2(\Omega;C([0,T]; L^2(D)^n))$ as well as in $L^2(\Omega;L^2(0,T;L^2(D)^{n\times n}))$, so using Theorem \ref{2scgrad} there exist $u \in  L^2(\Omega; L^2(0,T; H^1_0(D)^n))$ and $u_1 \in L^2(\Omega ; L^2(0,T; L^2(D;H^1_{\#}(Y))^n))$ such that 

\begin{equation}
\label{2s2}
\widetilde{u}^\e (\omega,t, x) \overset{2-s}{\longrightarrow} u (\omega,t,x) \mathds{1}_{Y^*}(y)\ \ \ \mbox{ in } \Omega\times [0,T]\times D,
\end{equation}
and
\begin{equation}
\label{2s3}
\widetilde{\nabla u}^\e (\omega,t, x) \overset{2-s}{\longrightarrow} \left(\nabla_x u(\omega,t, x) + \nabla_y u_1 (\omega,t,x,y) \right)\mathds{1}_{Y^*}(y)\ \ \ \mbox{ in } \Omega\times [0,T]\times D.
\end{equation}

We use Theorem \ref{2scint} to obtain as consequences of \eqref{2s2} and \eqref{2s3} the following two scale convergences:

\begin{equation}
\label{2s4}
\widetilde{U}^\e (\omega,t, x) \overset{2-s}{\longrightarrow} U (\omega,t,x) \mathds{1}_{Y^*}(y)\ \ \ \mbox{ in } \Omega\times [0,T] \times D,
\end{equation}
and

\begin{equation}
\label{2s5}
\widetilde{\nabla U}^\e (\omega,t, x) \overset{2-s}{\longrightarrow} \left(\nabla_x U(\omega,t, x) + \nabla_y U_1 (\omega,t,x,y) \right)\mathds{1}_{Y^*}(y)\ \ \ \mbox{ in } \Omega\times [0,T] \times D,
\end{equation}
where $U (\omega,t,x)=\displaystyle\int_0^t u (\omega,s,x) ds$ and $U_1 (\omega,t,x,y)=\displaystyle\int_0^t u_1 (\omega,s,x,y) ds$.

Also \eqref{estPet} gives us the existence of $P \in L^2(\Omega\times[0,T]\times D \times Y)^n$ such that:

\begin{equation}
\label{2s6}
\widetilde{P}^\e (\omega,t, x) \overset{2-s}{\longrightarrow} P(\omega,t,x,y) \mathds{1}_{Y^*}(y)\ \ \ \mbox{ in } \Omega\times[0,T]\times D.
\end{equation}

Now we are able to pass to the limit in \eqref{varforPt}. We will take a test function of the form $$\phi^\e(\omega,t,x)=\phi(\omega,t,x)+\e\psi(\omega,t,x,\dfrac{x}{\e}),$$ and we use the decompositions
\begin{equation}
\label{testf1}
\phi(\omega,t,x) = \phi_1(t)\phi_2(\omega) \phi_3(x),
\end{equation}
where $\phi_1 \in C_0^\infty (0,T)$, $\phi_2\in L^\infty(\Omega)$ and $\phi_3\in C_0^\infty (D)^n$ and
\begin{equation}
\label{testf2}
\psi(\omega, t, x,y) = \psi_1(t)\psi_2(\omega) \psi_3(x) \psi_4(y),
\end{equation}
where $\psi_1 \in C_0^\infty (0,T)$, $\psi_2\in L^\infty(\Omega)$, $\psi_3\in C_0^\infty (D)^n$ and $\psi_4\in C_{\#}^\infty(Y)$. We obtain after integrating over $\Omega\times[0,T]$:
\begin{equation}
\label{effeq1}
\begin{split}
&\int_{\Omega}\int_0^T\int_{D} \wt{u}^\e (\omega,t,x) \phi^\e(\omega,t,x) dx dt d\mathbb{P} +\e^2 \int_{\Omega}\int_0^T\int_{\partial O^\e} u^\e (\omega,t,x') \phi^\e(\omega,t,x') d\sigma(x')dt d\mathbb{P}-\\
&\int_{\Omega}\int_0^T\int_{D} \wt{u}_0^\e (\omega,x) \phi^\e(\omega,t,x) dx dt d\mathbb{P}-\e^2 \int_{\Omega}\int_0^T\int_{\partial O^\e} v_0^\e  (x') \phi^\e(\omega,t,x')d\sigma(x')dt d\mathbb{P}+\\
&\int_{\Omega} \int_0^T\int_{D} \int_0^t \nu  \wt{\nabla u}^\e (\omega,s,x) \nabla \phi^\e(\omega,t,x) ds dx dt d\mathbb{P}+\\
&\int_{\Omega}\int_0^T \int_{\partial O^\e} \int_0^t\e b u^\e (\omega,s,x') \phi^\e(\omega,t,x') ds d\sigma(x')dt d\mathbb{P} =\\
&\int_{\Omega}\int_0^T\int_{D} \wt{P}^\e(\omega,t,x)  \operatorname{div} \phi^\e(\omega,t,x) dx dt d\mathbb{P} + \int_{\Omega}\int_0^T \int_{D^\e} \int_0^t f (s,x) \phi^\e(\omega,t,x)ds dx dt d\mathbb{P}+ \\
&\int_{\Omega} \int_0^T\int_{D^\e} \int_0^t g_1 (s) dW_1(s) \phi^\e(\omega,t,x) dx  dt d\mathbb{P}+\\
& \int_{\Omega}\int_0^T  \int_{\partial O^\e} \int_0^t \e g_2^\e (s)dW_2(s) \phi^\e(\omega,t,x')d\sigma(x')dt d\mathbb{P}.
\end{split}
\end{equation}
We will compute the limit when $\e \to 0$ in \eqref{effeq1} term by term:
\\

As a consequence of \eqref{2s2} and \eqref{2s1} which implies the weak convergences in $L^2(\Omega \times [0,T] \times D)^n$ of the sequences $\widetilde{u}^\e(\omega,t,x)$ and $\widetilde{u}_0^\e(\omega,x)$ to $\displaystyle\int_Y \ u (\omega,t,x) \mathds{1}_{Y^*} (y) dy$ and to $ \displaystyle\int_Y u_0 (\omega,x)\mathds{1}_{Y^*} (y) dy$ we have:
\begin{equation}
\label{limit1}
\lim_{\e\to 0} \int_\Omega \int _0^T \int_{D} \wt{u}^\e (\omega,t,x) \phi^\e(\omega,t,x) dx dt d\mathbb{P} = |Y^*| \int_\Omega\int_0^T \int_D u (\omega,t,x) \phi(\omega,t,x) dx dt d\mathbb{P},
\end{equation}
and

\begin{equation}
\label{limit3}
\lim_{\e\to 0} \int_{\Omega}\int _0^T\int_{D} \wt{u}_0^\e (\omega,x) \phi^\e(\omega,t,x) dx dt d\mathbb{P} = |Y^*| \int_\Omega\int _0^T \int_D u_0 (\omega,x) \phi(\omega,t,x) dx dt d\mathbb{P}.
\end{equation}

The estimates \eqref{u0}, \eqref{uniform1}, \eqref{uniform2} imply that:

\begin{equation}
\label{limit2&4}
\e^2 \int_{\Omega}\int _0^T\int_{\partial O^\e} u^\e (\omega,t,x) \phi^\e(\omega,t,x')d\sigma(x')d\mathbb{P}= \e^2 \int_{\Omega}\int _0^T\int_{\partial O^\e} v_0^\e  (x) \phi^\e(\omega,t,x')d\sigma(x')dt d\mathbb{P} = 0.
\end{equation}

We use the two scale convergences \eqref{2s5} and \eqref{2s6} and that $\dfrac{\partial}{\partial x_i} \phi^\e = \dfrac{\partial \phi}{\partial x_i} + \e \dfrac{\partial \psi}{\partial x_i} + \dfrac{\partial \psi}{\partial y_i}$ to deduce that

\begin{equation}
\label{limit5}
\begin{split}
\lim_{\e\to 0} &\int_{\Omega}\int_0^T\int_{D} \nu  \wt{\nabla U}^\e (\omega,t,x) \nabla \phi^\e(\omega,t,x) dxdtd\mathbb{P} =\\
&\int_\Omega\int_0^T\int_D \int_{Y^*} \nu \left(\nabla_x U (\omega,t,x) + \nabla_y U_1 (\omega,t,x,y)  \right) \left(\nabla_x \phi (\omega,t,x) + \nabla_y \psi (\omega,t,x,y)\right) dy dxdt d\mathbb{P},
\end{split}
\end{equation}
and

\begin{equation}
\label{limit7}
\begin{split}
\lim_{\e\to 0} &\int_{\Omega}\int_0^T\int_{D} \wt{P}^\e (\omega,t,x) \operatorname{div} \phi^\e (\omega,t, x) dx dt d\mathbb{P} =\\
&\int_\Omega \int_0^T \int_D \int_{Y^*} P (\omega,t,x) \left(\operatorname{div}_x \phi(\omega,t,x) + \operatorname{div}_y \psi(\omega,t,x,y)\right) dy dx dt d\mathbb{P}.
\end{split}
\end{equation}

Now, if we fix $\omega\in\Omega$ and $t\in[0,T]$ the sequence of integrals $$\displaystyle\int_{D^\e} \int_0^t f(s,x) \phi^\e (\omega,t,x,\frac{x}{\e}) dsdx = \displaystyle\int_{D} \int_0^t f(s,x) \phi^\e (\omega,t,x,\frac{x}{\e})\mathds{1}_{D^\e}(x) dsdx$$ converges using formula (5.5) from \cite{A-2s} to $$\displaystyle |Y^*|\int_D \int_0^t f(s,x) \phi (\omega,t,x) dsdx,$$ so after applying Vitali's theorem we obtain:
\begin{equation}
\label{limit8}
\lim_{\e\to 0} \int_\Omega\int_0^T \int_{D^\e} \int_0^t  f(s,x) \phi^\e(\omega,t,x) ds dx dt d\mathbb{P} = |Y^*|\int_\Omega \int_0^T \int_D  \int_0^t f(s,x) \phi(\omega,t,x) ds dx dt d\mathbb{P}.
\end{equation}

Now we compute the limits of the stochastic integrals and of the integrals over the boundary $\partial O^\e$, and the results we obtain are shown in the next three Lemmas.
\begin{lemma}
\label{lemmalimit9}
\begin{equation}
\label{limit9}
\lim_{\e\to 0} \int_{\Omega}\int_0^T \int_{D^\e} \int_0^t g_1 (s) dW_1(s) \phi^\e(\omega,t,x) dx dt d\mathbb{P} = |Y^*|\int_\Omega \int_0^T \int_D  \int_0^t g_1 (s) dW_1(s) \phi(\omega,t,x) dx dt d\mathbb{P}.
\end{equation}
\end{lemma}
\begin{proof}
We substitute $\phi^\e(\omega,t,x) = \phi(\omega,t,x) + \e \psi(\omega,t,x,\frac{x}{\e})$ and use the decomposition \eqref{testf1} of $\phi^\e$ and we get that \eqref{limit9} can be rewritten as

\begin{equation*}
\lim_{\e\to 0}\int_{\Omega}\phi_1(\omega) \int_0^T \phi_2(t)\int_{D} \int_0^t g_1 (s) dW_1(s) \left(\phi_3(x) (\mathds{1}_{D^\e}(x)-|Y^*|) + \e \psi(\omega, t,x ,\frac{x}{\e})\right) dx dt d\mathbb{P} =0.
\end{equation*}
Now using the H{\"o}lder's inequality, and the fact that $||\psi(\omega,t,x,\frac{x}{\e})||_{L^2(\Omega\times[0,T]\times D)^n}$ is a sequence uniformly bounded in $\e$ by $||  \psi(\omega,t,x,y) ||_{L^2(\Omega\times[0,T]\times D;C_{\#}(Y))^n}$, we obtain:

\begin{equation*}
\begin{split}
&\lim_{\e\to 0}\left|\int_{\Omega}\int_0^T\int_{D} \int_0^t g_1 (s) dW_1(s) \e \psi(\omega, t,x ,\frac{x}{\e}) dx dt d\mathbb{P}\right| \leq \\
&\lim_{\e\to 0} \e T^2 ||\psi(\omega,t,x,\frac{x}{\e})||_{L^2(\Omega\times[0,T]\times D)^n} \sup_{t\in[0,T]} ||g_1(t)||^2_{Q_1} = 0.
\end{split}
\end{equation*}

Using H{\"o}lder's inequality again:
\begin{equation*}
\begin{split}
&\lim_{\e\to 0}\left|\int_{\Omega}\phi_1(\omega) \int_0^T \phi_2(t)\int_{D} \int_0^t g_1 (s) dW_1(s) \phi_3(x) (\mathds{1}_{D^\e}(x)-|Y^*|) dx dt d\mathbb{P} \right|\leq \\
&\lim_{\e\to 0} ||\phi_1 ||_{L^2(0,T)} ||\phi_2||_{L^2(\Omega)} \int_\Omega\int_0^T \left| \int_D \int_0^t g_1 (s) dW_1(s) \phi_3(x) (\mathds{1}_{D^\e}(x)-|Y^*|) dx \right|^2 dt d\mathbb{P}.
\end{split}
\end{equation*}

We rewrite the integral as follows:
\begin{equation*}
\begin{split}
&\int_{\Omega} \int_0^T |\int_{D} \sum_{i=1}^\infty \sqrt{\lambda_{i1}}\int_0^t g_1 (s) e_{i1}(x) d\beta_i(s) \phi_3(x) (\mathds{1}_{D^\e}(x)-|Y^*|)dx|^2 dt d\mathbb{P}.
\end{split}
\end{equation*}
where $(\beta_i)_{i=1}^\infty$ is a sequence of real valued independent Brownian motions, and $(e_{i1})_{i=1}^\infty$  and  $(\lambda_{i1})_{i=1}^\infty$ are previously defined in \eqref{eqCT}.

We apply the stochastic Fubini theorem and the It\^o's isometry and the sum becomes
\begin{equation*}
\begin{split}
&\sum_{i=1}^\infty \lambda_{i1}  \int_{\Omega}\int_0^T |\int_0^t d\beta_i(s) \int_{D} g_1 (s) e_{i1}(x) (\mathds{1}_{D^\e}(x)-|Y^*|) \phi_2(x) dx|^2  dt d\mathbb{P}=\\
&\sum_{i=1}^\infty \lambda_{i1} \int_0^T |\int_0^t ds \int_{D} g_1 (s) e_{i1}(x) (\mathds{1}_{D^\e}(x)-|Y^*|) \phi_2(x) dx|^2  dt \leq \\
&C \sum_{i=1}^N \lambda_{i1} |\int_0^t ds \int_{D} g_1 (s) e_{i1}(x) (\mathds{1}_{D^\e}(x)-|Y^*|) \phi_2(x) dx|^2 + C \sum_{i=N}^\infty \lambda_{i1} ||g_1 (s) e_{i1}(x)||^2_{L^2(D)^n},
\end{split}
\end{equation*}
for any $N \in \mathbb{Z}_+^*$, and for a constant $C$ independent of $N$. The first sum goes to $0$ from dominated convergence theorem and the weak convergence to $0$ in $L^2(D)$ of $(\mathds{1}_{D^\e}(x)-|Y^*|)\phi_2(x)$. The second term goes to $0$ when $N \to \infty$ because of \eqref{eqCT}. Hence, \eqref{limit9} follows.
\end{proof}
We compute the limits that involve integrals over the boundaries, and we will make use of the techniques from \cite{CDE96} that give a way of transforming integrals over the surface in integrals over the volume.
\begin{lemma}
\label{lemmalimit6}
\begin{equation}
\label{limit6}
\begin{split}
\lim_{\e\to 0} & \int_{\Omega}\int_0^T \int_{\partial O^\e} \int_0^t\e b u^\e (\omega,s,x') \phi^\e(\omega,t,x') dsd\sigma(x') dt d\mathbb{P} =\\
&|\partial O|\int_{\Omega} \int_0^T\int_D \int_0^t b u (\omega,s,x) \phi(\omega,t,x) ds dx dt d\mathbb{P} .
\end{split}
\end{equation}
\end{lemma}
\begin{proof}
We will define as in \cite{CDE96} the solution of the following system:
\begin{equation}
\label{w1}
\left\{
\begin{array}{rll}
-\Delta w_1&= -  \dfrac{|\partial O|}{|Y^*|} &\mbox{ in }\  Y^* , \\
\displaystyle \frac{\partial w_1}{\partial n} &=1&\mbox{ on }\  \partial O, \\
\displaystyle \int_{Y^*} w_1 &=0,& \\
w_1&- Y-periodic,& \\
\end{array}
\right.
\end{equation}
which, according to Remark 4.3 from \cite{CDE96} belongs to $W^{1,\infty}(Y^*)$. We define $w_1^\e(x) = \e^2 w_1(\frac{x}{\e})$, and straightforward calculations show that $w_1^\e$ satisfies:
\begin{equation}
\label{w1e}
\left\{
\begin{array}{rll}
-\Delta w_1^\e&= -\dfrac{|\partial O|}{|Y^*|} &\mbox{ in }\  D^\e , \\
\displaystyle \frac{\partial w_1^\e}{\partial n} &=\e&\mbox{ on }\  \partial O^\e. \\
\end{array}
\right.
\end{equation}
For fixed $\omega\in\Omega$ and $t\in[0,T]$, the sequence $v^\e(w,t,\cdot) = b \displaystyle\int_0^t u^\e (\omega,s,\cdot) \phi^\e(\omega,t,\cdot) ds$ belongs to $W^{1,1} (D^\e)$ and we have:
\begin{equation*}
\begin{split}
&\int_{\partial O^\e} \e v^\e (\omega,t,x') d\sigma(x')  =
\int_{\partial O^\e} \frac{\partial w_1^\e}{\partial n} (x') v^\e (\omega,t,x')d\sigma(x') =\\
&\int_{D^\e} \left(\Delta w_1^\e (x) v^\e (\omega,t,x) + \nabla w_1^\e(x) \nabla v^\e (\omega,t,x)\right) dx.
\end{split}
\end{equation*}
But from \eqref{uniform1} and the definitions of $w_1^\e$ and $v^\e$,  $$|\int_\Omega\int_0^T\int_{D^\e} \nabla w_1^\e(x) \nabla v^\e (\omega,t,x) dx dt d\mathbb{P}| \to 0,$$
and $$\int_\Omega \int_0^T \int _{D^\e} v^\e (\omega,t,x) dx dt d\mathbb{P} \to |Y^*| b \int_\Omega \int _D \int_0^t u (\omega,s,x) \phi(\omega,t,x) ds dx dt d\mathbb{P},$$
which implies \eqref{limit6}.
\end{proof}
\begin{lemma}
\label{lemmalimit10}
\begin{equation*}
\begin{split}
&\lim_{\e\to 0} \int_{\Omega} \int_0^T \int_{\partial O^\e} \int_0^t \e g_2^\e (s)dW_2(s) \phi^\e(\omega,t,x') d\sigma(x')dt d\mathbb{P}=\\
&|\partial O|\int_{\Omega}  \int_0^T \int_D \int_0^t g_{21} (s) dW_2(s) \phi(\omega,t,x) dx dt d\mathbb{P} + \int_\Omega \int_0^T\int_{\partial O} \int_0^t  g_{22}(s) dW_2(s) d\sigma \int_D \phi (\omega,t,x)dxdtd\mathbb{P}.
\end{split}
\end{equation*}
\end{lemma}
\begin{proof}
Similar arguments discussed in \eqref{limit6} and \eqref{limit9} will be used here. We will prove:
\begin{equation}
\label{limit10'}
\begin{split}
\lim_{\e\to 0} &\int_{\Omega}  \int_0^T \int_{\partial O^\e} \int_0^t \e g_{21} (s) dW_2(s) \phi^\e(\omega,t,x')d\sigma(x')dt d\mathbb{P}=\\
&|\partial O|\int_{\Omega}  \int_0^T \int_D \int_0^t g_{21} (s) dW_2(s) \phi(\omega,t,x) dx dt d\mathbb{P},
\end{split}
\end{equation}
and
\begin{equation}
\label{limit10''}
\begin{split}
\lim_{\e\to 0} &\int_{\Omega}  \int_0^T \int_{\partial O^\e} \int_0^t \e \mathcal{R}^\e g_{22} (s) dW_2(s) \phi^\e(\omega,t,x')d\sigma(x')dt d\mathbb{P} = \\
&\int_\Omega \int_0^T\int_{\partial O} \int_0^t  g_{22}(s) dW_2(s) d\sigma \int_D \phi (\omega,t,x)dxdtd\mathbb{P}.
\end{split}
\end{equation}
To show \eqref{limit10'} we use the functions $w_1$ and $w_1^\e$ defined in \eqref{w1} and \eqref{w1e}:
\begin{equation*}
\begin{split}
& \int_{\partial O^\e} \int_0^t \e g_{21} (s) dW_2(s) \phi^\e(\omega,t,x') d\sigma(x') = \int_{\partial O^\e} \frac{\partial w_1^\e}{\partial n}\int_0^t g_{21} (s) dW_2(s) \phi^\e(\omega,t,x') d\sigma(x') =\\
& \int_{D^\e} \Delta w_1^\e(x) \int_0^t g_{21} (s) dW_2(s) \phi^\e(\omega,t,x') d\sigma(x') +\int_{D^\e} \nabla w_1^\e(x) \int_0^t g_{21} (s) dW_2(s)\nabla_x\phi^\e(\omega,t,x) dx=\\
&\frac{|\partial O|}{|Y^*|} \int_{D^\e} \int_0^t g_{21} (s) dW_2(s) \phi^\e(\omega,t,x') d\sigma(x') +\int_{D^\e} \nabla w_1^\e(x) \int_0^t g_{21} (s)dW_2(s) \nabla_x\phi^\e(\omega,t,x) dx.
\end{split}
\end{equation*}
But
$$\lim_{\e\to 0} \int_\Omega  \int_0^T \int_{D^\e} \nabla w_1^\e(x) \int_0^t g_{21} (s)dW_2(s) \nabla_x\phi^\e(\omega,t,x) dx dt d\mathbb{P} =0 $$
follows from the condition \eqref{eqCT} and the definiton of $w_1^\e$.

Now, the same computation used to show \eqref{limit9} yields:
\begin{equation*}
\begin{split}
\lim_{\e\to 0} &\int_\Omega  \int_0^T \int_{D^\e} \int_0^t g_{21} (s) dW_2(s) \phi^\e(\omega,t,x)dxdt d\mathbb{P} =\\
|&Y^*|\int_\Omega  \int_0^T \int_D \int_0^t g_{21} (s) dW_2(s) \phi(\omega,t,x)dxdt d\mathbb{P}.
\end{split}
\end{equation*}

To show \eqref{limit10''}, let $(e_{i2}(x))_{i=1}^\infty$ and  $(\lambda_{i2})_{i=1}^\infty$ 
previously defined in \eqref{eqCT}. Denote by $h_i(s)=g_{22}(s) e_{i2} \in L^2(\partial O)$, 
for each $i\in \mathbb{Z}_+$ and $s\in [0,T]$. We infer from \eqref{eqCT} that
\begin{equation}
\label{hq}
\sup_{s\in[0,T]} \sum_{i=1}^\infty \lambda_{i2} ||h_i(s)||^2_{L^2(\partial O)} < +\infty.
\end{equation}
We will define $w_i(s)$ similarly as $w_1$ to be the unique element in $H^1(Y^*)$ that solves:
\begin{equation}
\label{wis}
\left\{
\begin{array}{rll}
-\Delta w_i(s)&= - \dfrac{1}{|Y^*|} \displaystyle\int_{\partial O} h_i(s) d\sigma &\mbox{ in }\  Y^* , \\
\displaystyle \frac{\partial w_i(s)}{\partial n} &=h_i(s)&\mbox{ on }\  \partial O, \\
\displaystyle \int_{Y^*} w_i(s) &=0,& \\
w_i(s)&- Y-periodic,& \\
\end{array}
\right.
\end{equation}
and $w_i^\e(s) = \e^2 w_i(s)\left( \dfrac{\cdot}{\e} \right)$ that will solve
\begin{equation}
\label{wise}
\left\{
\begin{array}{rll}
-\Delta w_i^\e(s)&=  - \dfrac{1}{|Y^*|} \displaystyle\int_{\partial O} h_i(s) d\sigma &\mbox{ in }\  D^\e , \\
\displaystyle \frac{\partial w_i^\e(s)}{\partial n} &=\e \mathcal{R}^\e h_i (s)&\mbox{ on }\  \partial O^\e. \\
\end{array}
\right.
\end{equation}

There exists a constant $C$ independent of $i$ and $s$ such that $||w_i(s)||_{H^1(Y^*)^n} \leq C ||h_i(s)||_{L^2(\partial O)^n}$ and using \eqref{hq} we also have:
\begin{equation}
\label{propwis}
\sup_{s\in[0,T]}\sum_{i=1}^\infty \lambda_{i2}||w_i(s)||^2_{H^1(Y^*)^n}  < \infty.\\
\end{equation}
Using the decomposition of $\phi^\e$\ and previously used computations, we have to show that:
\begin{equation*}
\int_{\Omega} \left|\int_{\partial O^\e} \int_0^t \e \mathcal{R}^\e g_{22} (s) dW_2(s) \phi_2(x') d\sigma(x') -\int_{\partial O} \int_0^t g_{22} (s) dW_2(s) d\sigma  \int_D\phi_2(x) dx \right|^2d\mathbb{P}
\end{equation*}
converges to $0$ uniformly for $t\in[0,T]$.
We perform the following calculations:

\begin{equation*}
\begin{split}
&\int_{\Omega}  \left|\int_{\partial O^\e} \int_0^t \e \mathcal{R}^\e g_{22} (s) dW_2(s) \phi_2(x') d\sigma(x') -\int_{\partial O} \int_0^t g_{22} (s) dW_2(s)d\sigma \int_D\phi_2(x) dx \right|^2d\mathbb{P}=\\
&\int_{\Omega}  \sum_{i=1}^\infty \lambda_{i2}\left|\int_{\partial O^\e} \int_0^t \e \mathcal{R}^\e g_{22} (s)e_{i2} d\beta_i(s) \phi_2(x') d\sigma(x') -\int_{\partial O} \int_0^t g_{22} (s)e_{i2} d\beta_i(s)d\sigma(x')  \int_D\phi_2(x) dx \right|^2d\mathbb{P}=\\
&\int_{\Omega}  \sum_{i=1}^\infty \lambda_{i2}\left|\int_0^t \int_{\partial O^\e}  \e h_i^\e(s) \phi_2(x')d\sigma(x')d\beta_i(s)  -\int_0^t\int_{\partial O} h_i(s) d\sigma  \int_D\phi_2(x) dxd\beta_i(s) \right |^2d\mathbb{P}=\\
 &\sum_{i=1}^\infty \lambda_{i2} \int_0^t ds \left|\int_{\partial O^\e}  \e h_i^\e(s) \phi_2(x')d\sigma(x')-\int_{\partial O} h_i(s) d\sigma  \int_D\phi_2(x) dx  \right|^2=\\
&\sum_{i=1}^\infty \lambda_{i2} \int_0^t ds \left|\int_{\partial O^\e}  \frac{\partial w_i^\e(s)}{\partial n} \phi_2(x')d\sigma(x')-\int_{\partial O} h_i(s)d\sigma \int_D\phi_2(x) dx  \right|^2=\\
&\sum_{i=1}^\infty \lambda_{i2} \int_0^t ds \left|\int_{D^\e}  \left(\Delta w_i^\e(s) \phi_2(x) 
+\nabla w_i^\e(s) \nabla \phi_2(x) \right)dx -\int_{\partial O} h_i(s) d\sigma \int_D\phi_2(x) dx  \right|^2
=\\
&\sum_{i=1}^\infty \lambda_{i2} \int_0^t ds \left|\int_{D}\int_{\partial O} h_i(s) d\sigma \left(\frac{1}{|Y^*|} \mathds{1}_{D^\e}(x)  -1\right)\phi_2(x) dx +\int_{D^\e}\nabla w_i^\e(s) \nabla \phi_2(x) dx \right|^2\leq\\
&C\sum_{i=1}^\infty \lambda_{i2} \int_0^t ds \left|\int_{D}\int_{\partial O} h_i(s) d\sigma\left(\mathds{1}_{D^\e}(x)  -|Y^*|\right)\phi_2(x) dx \right|^2 + C\sum_{i=1}^\infty \lambda_{i2} \int_0^t ds \left|\int_{D^\e}\nabla w_i^\e(s) dx \right|^2.
\end{split}
\end{equation*}

The second term of the sum will be bounded by 
$$C\e^2 \displaystyle\sum_{i=1}^\infty \lambda_{i2} \int_0^T ds \left|\int_{Y^*}\nabla w_i(s) dx\right|^2.$$ Using \eqref{propwis} it converges to $0$ uniformly for $t\in[0,T]$. 

We decompose the first term similarly to \eqref{limit9}, into

\begin{equation*}
\begin{split}
&\sum_{i=1}^N \lambda_{i2} \int_0^t ds \left|\int_{D}\int_{\partial O} h_i(s) d\sigma \left(\mathds{1}_{D^\e}(x)  -|Y^*|\right)\phi_2(x) dx \right|^2 +\\
& \sum_{i=N+1}^\infty \lambda_{i2} \int_0^t ds \left|\int_{D}\int_{\partial O} h_i(s) d\sigma \left(\mathds{1}_{D^\e}(x)  -|Y^*|\right)\phi_2(x) dx \right|^2.
\end{split}
\end{equation*}

The first sum goes to $0$ for any fixed $N$ because of the weak convergence to $0$ in $L^2(D)$ of $(\mathds{1}_{D^\e}(x)-|Y^*|)\phi_2(x)$. The second sum goes to $0$ when $N \to \infty$ because of \eqref{hq}.

Using the limits obtained in the variational formulation we get

\begin{equation}
\label{effeq2}
\begin{split}
&|Y^*| \int_\Omega\int_0^T\int_D u (\omega,t,x) \phi(\omega,t,x) dx dt d\mathbb{P} - |Y^*| \int_\Omega\int_0^T\int_D u_0 (\omega,x) \phi(\omega,t,x) dx dt d\mathbb{P}+\\
&\int_\Omega\int_0^T\int_D \int_{Y^*} \int_0^t \nu \left(\nabla_x u (\omega,s,x) + \nabla_y u_1 (\omega,s,x,y)  \right) \left(\nabla_x \phi (\omega,t,x) + \nabla_y \psi (\omega,t,x,y)\right) ds dy dx dtd\mathbb{P}-\\
&\int_\Omega\int_0^T \int_D \int_{Y^*} P (\omega,t,x,y) \left(\operatorname{div}_x \phi(\omega,x) + \operatorname{div}_y \psi(\omega,t,x,y)\right) dy dx dt d\mathbb{P}+\\
&|\partial O|\int_\Omega\int_0^T\int_D \int_0^t b u^\e (\omega,s,x) \phi(\omega,t,x) ds dx dt d\mathbb{P} =\\
&|Y^*|\int_\Omega\int_0^T \int_D  \int_0^t f(s,x) \phi(\omega,t,x) ds dx dt d\mathbb{P} + |Y^*|\int_\Omega\int_0^T \int_D  \int_0^t g_1 (s) dW_1(s) \phi(\omega,t.x) dx dt d\mathbb{P} +\\ 
&|\partial O|\int_\Omega\int_0^T\int_D \int_0^t g_{21} (s) dW_2(s) \phi(\omega,t,x) dx dt d\mathbb{P} +\\
&\int_\Omega\int_0^T\int_{\partial O} \int_0^t  g_{22}(s) dW_2(s) d\sigma\int_D \phi (\omega,t,x)dxdtd\mathbb{P}.
\end{split}
\end{equation}
\end{proof}

\subsection{The cell problem}
\

We separate the terms in \eqref{effeq2} according to $\phi$ and $\psi$ and consider their decompositions introduced in subsection \ref{subsection51}, which means that $U_1$ and $P$ satisfy the following equation for every $t\in[0,T]$, almost every $x\in D$ and $\omega\in\Omega$:

\begin{equation}
\label{U11}
\int_{Y^*} \nu\left(\nabla_x U (\omega,t,x) + \nabla_y U_1 (\omega,t,x,y)  \right) \nabla_y\psi_4(y) - P(\omega,t,x,y) \operatorname{div}_y \psi_4(y) dy =0,
\end{equation}
for every $\psi_4 \in C^\infty_{\#}(Y)^n$.

We have also, as a consequence of  \eqref{2s5} that 

\begin{equation}
\label{U12}
\operatorname{div}_x U(\omega,t,x)+\operatorname{div}_y U_1(\omega,t,x,y)=0,
\end{equation}
being the two scale limit of the sequence $\widetilde{\operatorname{div}_x U}^\e(\omega,t,x)$.

Equations \eqref{U11}--\eqref{U12} lead us to define for every $\Lambda \in \mathbb{R}^{n\times n}$, $w_\Lambda \in H^1_{\#}(Y^*)^n$ and $q_\Lambda\in L^2_{\#} (Y^*)$, as the solution of the following problem in $Y^*$:

\begin{equation}
\left\{
\begin{array}{rll}
\displaystyle\int_{Y^*} \nu\nabla_y \left( \Lambda y+w_\Lambda (y)  \right) \nabla_y\phi(y) - q_\Lambda (y) \operatorname{div}_y \phi(y) dy &=0 &\mbox{ for all } \phi \in H^1_{\#}(Y)^n\\
\operatorname{div}_y \left( \Lambda y + w_\Lambda (y) \right) & =0 &\mbox{ in } Y^*
\end{array}
\right.
\end{equation}
which is equivalent to the system:

\begin{equation}
\label{cellsystem}
\left\{
\begin{array}{rll}
-\nu\Delta w_\Lambda (y) + \nabla q_\Lambda (y) &=0 &\mbox{ in } Y^*\\

\operatorname{div}_y \left( \Lambda y + w_\Lambda (y) \right) & =0 &\mbox{ in } Y^*\\

\nu\dfrac{\partial \left(\Lambda y + w_\Lambda (y)\right)}{\partial n} - q_\Lambda (y) n&=0 & \mbox{ on } \partial O\\

w_\Lambda,\ q_\Lambda\  &Y-periodic\\

\displaystyle\int_{Y^*} w_\Lambda (y) dy &=0
\end{array}
\right.
\end{equation}

Define $\mathcal{C}$ the linear form on $\mathbb{R}^{n\times n}$, by:

\begin{equation}
\label{defC}
\mathcal{C} \Lambda = \int_{Y^*} \left(\nu\left(\Lambda+\nabla_y w_\Lambda (y)\right) -  q_\Lambda (y) I \right)dy,
\end{equation}
where $I \in \mathbb{R}^{n\times n}$ is the identity matrix.

Let $\mathbf{e}_{ij} \in \mathbb{R}^{n\times n}$ be defined by $(\mathbf{e}_{ij})_{kh}=\delta_{ik} \delta_{jh}$, for every $1 \leq k, h \leq n$, and let $w_{ij}$ and $q_{ij}$ be the solutions of the cell problem \eqref{cellsystem} corresponding to $\Lambda = \mathbf{e}_{ij}$. Then, by linearity:
\begin{equation}
\label{sollambda}
w_{\Lambda}(y)=\sum_{i, j=1}^n w_{ij}(y) \Lambda_{ij}\ \mbox{  and  }\ q_\Lambda (y)=\sum_{i, j=1}^n q_{ij} (y) \Lambda_{ij},
\end{equation}
and the linear form $\mathcal{C}$ can be given componentwise:
\begin{equation}
\label{defC1}
(\mathcal{C} \mf{e}_{ij})_{kh} = \int_{Y^*} \left( \nu\delta_{ik} \delta_{jh} + \nu\dfrac{\partial w_{ij}^k}{\partial y_h} (y) -q_{ij}\delta_{kh}\right) dy.
\end{equation}
\subsection{The homogenized system}
\label{subsection53}
\

According to  \eqref{sollambda} the solutions $U_1$ and $P$ of  \eqref{U11}--\eqref{U12} are given by
\begin{equation}
\label{U1}
U_1(\omega,t,x,y)=\sum_{i,j=1}^n w_{ij}(y) \dfrac{\partial (U)_i}{\partial x_j}(\omega,t,x),
\end{equation}

\begin{equation}
\label{P}
P(\omega,t,x,y)=\sum_{i,j=1}^n q_{ij}(y) \dfrac{\partial (U)_i}{\partial x_j}(\omega,t,x),
\end{equation}
and the equation \eqref{effeq2} becomes

\begin{equation}
\label{effeq3}
\begin{split}
&|Y^*| \int_D u (\omega,t,x) \phi_3(x) dx - |Y^*| \int_D u_0 (\omega,x) \phi_3(x) dx +\int_D \left( \mathcal{C} \nabla U (\omega,t,x) \right) \nabla \phi_3(x) dx-\\
&|\partial O|\int_D \int_0^t b u^\e (\omega,s,x) \phi_3(x) ds dx =|Y^*|\int_D  \int_0^t f(s,x) \phi_3(x) ds dx +\\
&|Y^*| \int_D  \int_0^t g_1 (s) dW_1(s) \phi_3(x) dx +|\partial O| \int_D \int_0^t g_{21} (s) dW_2(s) \phi_3(x) dx + \\
&\int_{\partial O} \int_0^t  g_{22}(s) dW_2(s) d\sigma \int_D \phi_3(x)dx.
\end{split}
\end{equation}
for every $\phi_3 \in C^\infty_0(D)^n$, for every $t\in[0,T]$ and a.s. $\omega\in\Omega$, which implies that $u$ is the solution of the following stochastic partial differential equation:

\begin{equation}
\label{eqhom}
\left\{
\begin{array}{rll}
d u (t) &=\left[\operatorname{div}_x \left(\mathcal{C} \nabla_x u (t) \right)+\dfrac{|\partial O|}{|Y^*|} b u (t)+f (t)\right] dt &\\
&+g_1 (t) d W_1(t) + \dfrac{1}{|Y^*|}\left( |\partial O| g_{21} (t)+\displaystyle\int_{\partial O} g_{22}(t)d\sigma \right) d W_2(t) &\mbox{in}\  D, \\
u(0)&=u_{0}&\mbox{in}\  D. \\
\end{array}
\right.
\end{equation}

If we denote by $u^*$ the weak limit in $L^2(\Omega\times[0,T]\times D)^n$ of the sequence $\wt{u}^\e$, according to \eqref{2s2}, then it will solve a similar equation:

\begin{equation}
\label{eqhom'}
\left\{
\begin{array}{rll}
d u^* (t) &=\left[\operatorname{div}_x \left(\mathcal{C} \nabla_x u^* (t) \right)+|\partial O| b u^* (t)+|Y^*|f (t)\right] dt &\\
&+|Y^*|g_1 (t) d W_1(t) + \left( |\partial O| g_{21} (t)+\displaystyle\int_{\partial O} g_{22}(t)d\sigma\right) d W_2(t) &\mbox{in}\  D, \\
u^*(0)&=|Y^*|u_{0}&\mbox{in}\  D. \\
\end{array}
\right.
\end{equation}

%\begin{equation}
%u_1 (t,x,y)=\sum_{i, j=1}^n w_{ij} (y) (\nabla_x u (t,x) )_{ij}\ \mbox{  and  }\ p(t,x,y)=\sum_{i, %j=1}^n q_{ij} (y) (\nabla_x u (t,x) )_{ij}.
%\end{equation}

Let us define now the operator $A^* : D(A^*)\subset L^2(D)^n\mapsto L^2(D)^n$, by $A^* u = -\operatorname{div} \left(\mathcal{C} \nabla u\right)$ for every $u\in D(A^*)$, where
$$D(A^*)=H_0^1(D)^n\cap H^2(D)^n,$$
and $\mathcal{C}$ is introduced in \eqref{defC}.

\begin{lemma} (Properties of the operator $A^*$) The linear operator $A^*$ is positive and self-adjoint in $L^2(D)^n$.
\end{lemma}
\begin{proof}
All we have to show is that the operator $\mathcal{C}$ defined in \eqref{defC} is symmetric and positive definite. Let $\Lambda_1$ and $\Lambda_2$ be two matrices from $\mathbb{R}^{n\times n}$ and let us compute $\mathcal{C}\Lambda_1\cdot\Lambda_2$:
\begin{equation*}
\mathcal{C}\Lambda_1\cdot\Lambda_2=\int_{Y^*} \left(\nu\left(\Lambda_1+\nabla_y w_{\Lambda_1} (y)\right) -  q_{\Lambda_1} (y) I \right)\cdot \Lambda_2dy.
\end{equation*}
Using $w_{\Lambda_2}$ as a test function in \eqref{cellsystem}, for $\Lambda=\Lambda_1$ we get that:
\begin{equation*}
\int_{Y^*} \left(\nu\left(\Lambda_1+\nabla_y w_{\Lambda_1} (y)\right) -  q_{\Lambda_1} (y) I \right)\cdot \nabla w_{\Lambda_2}dy=0,
\end{equation*}
and given that $\operatorname{div}_y q_{\Lambda_1}(y) = \operatorname{div}_y q_{\Lambda_2}(y)=0$ in $Y^*$,
\begin{equation*}
\int_{Y^*} \left(\nu\nabla_y \left(\Lambda_1 y+w_{\Lambda_1} (y)\right)  \right)\cdot q_{\Lambda_2} (y)dy=\int_{Y^*} \left(\nu\nabla_y \left(\Lambda_2 y+w_{\Lambda_2} (y)\right)  \right)\cdot q_{\Lambda_1} (y)dy=0.
\end{equation*}

Using these relations, elementary calculations will give us that
\begin{equation}
\label{Csym}
\mathcal{C}\Lambda_1\cdot\Lambda_2=\int_{Y^*} \nu\left(\Lambda_1+\nabla_y w_{\Lambda_1} (y)\right)\cdot\left(\Lambda_2+\nabla_y w_{\Lambda_2} (y)\right) dy,
\end{equation}
which implies that $\mathcal{C}$ is symmetric and positive definite. As a consequence, we have the following existence result:
\end{proof}

\begin{theorem}\label{mildhom}
 (Well posedness of the homogenized equation) Let $S^*(t)_{t\geq 0}$ be the semigroup generated by the operator $A^*$ and assume that assumption \eqref{eqCT} is satisfied. Then, the equation \eqref{eqhom'} admits a unique mild solution $u^* \in L^2(\Omega; C([0,T],L^2(D)^n)\cap L^2(0,T;H_0^1(D)^n))$, given by
\begin{equation}\label{mildhomeq}
\begin{split}
u^*(t)=&|Y^*| S^*(t)u^*_{0}+|\partial O|\int_{0}^{t}S^*(t-s)bu^*(s)ds+|Y^*|\int_{0}^{t} S^*(t-s)f(s)ds\\
+&|Y^*|\int_{0}^{t}S^*(t-s)g_1(s)dW_1(s)+|\partial O|\int_{0}^{t}S^*(t-s)g_{21}(s)dW_2(s)\\
+&\int_{0}^{t}S^*(t-s)\displaystyle\int_{\partial O} g_{22}(s)d\sigma dW_2(s), \quad t\in [0,T].
\end{split}
\end{equation}
\end{theorem}

\begin{remark}

\ 

\begin{enumerate}

\item The solution of \eqref{eqhom'} is not divergence free, so the homogenized equation does not contain a pressure term. However, the effect of the pressure $P^\e(t)$ appears implicitely through the matrix $\mathcal{C}$, hence the cell problem.

\item The overall effect of the boundary condition is seen in the homogenized equation through the Brinkman term $|\partial O| b u^*(t)$ and through an extra stochastic forcing acting in the volume, $\left( |\partial O| g_{21} (t)+\displaystyle\int_{\partial O} g_{22}(t)d\sigma\right) \dot {W_2}(t)$.

\item The convergence of the sequence $u^\e$ to the limit $u^*$ is strong in the probabilistic sense, but weak in the deterministic sense. In particular, the sequence $\wt{u}^\e$ of the extensions by $0$ of $u^\e$ inside $O^\e$ converges to $u^*$ weakly in $L^2(\Omega\times [0,T] \times D)^n$.

\end{enumerate}
\end{remark}

\end{document}